\documentclass[10pt]{amsart}

\usepackage{color}
\usepackage[colorlinks=true,linkcolor=blue]{hyperref} 

\usepackage{ulem}
\usepackage{graphics}
\usepackage{epsfig} 
\usepackage[colorlinks=true,linkcolor=blue]{hyperref} 
\newtheorem{theor}{{\bf Theorem}}[section]
\newtheorem{lemma}[theor]{{\bf Lemma}}           
\newtheorem{coro}[theor]{{\bf Corollary}}
\newtheorem{prop}[theor]{{\bf Proposition}}

\newtheorem{remark}{{\bf Remark}}[section]
\numberwithin{equation}{section}

\newcommand{\sauf}{\setminus}

\newcommand{\epsl}{\varepsilon}

\newcommand{\bfb}{\boldsymbol{b}}
\newcommand{\bfg}{\boldsymbol{g}}
\newcommand{\bfbeta}{\boldsymbol{\beta}}
\newcommand{\bfepsl}{\boldsymbol{\epsl}} 
\newcommand{\bfmu}{\boldsymbol{\mu}} 
\newcommand{\bfalpha}{\boldsymbol{\alpha}}

\newcommand{\calA}{\mathcal A}
\newcommand{\calB}{\mathcal B}
\newcommand{\calC}{\mathcal C}

\newcommand{\calH}{\mathcal H}
\newcommand{\calHD}{\calH^{D}}

\newcommand{\hatHD}{\hat{H}^{D}}

\newcommand{\Kint}{{\bf K}_{\rm int}}
\newcommand{\Kext}{{\bf K}_{\rm ext}}

\newcommand{\calP}{\mathcal P}
\newcommand{\calS}{\mathcal S}

\newcommand{\Com}{{\rm comat}}
\newcommand{\Ima}{{\rm Im}}

\newcommand{\rg}{{\rm rank}}
\newcommand{\rmd}{{\rm d}}

\renewcommand{\span}{{\rm span}}
\newcommand{\supp}{{\rm supp}\,}
\newcommand{\Om}{\Omega}

\newcommand{\om}{\omega}
\newcommand{\ove}{\overline}

\newcommand{\N}{\mathbf N}
\newcommand{\R}{\mathbf R}

\newcommand{\Sm}{\mathbb S}
\newcommand{\T}{\mathbb T}

\title{Rellich Type Theorem and unique continuation property for Discrete Maxwell Operators}
\author{Hiroshi Isozaki$^{\small 1}$}
\author{Olivier Poisson$^{\small 2}$}

\begin{document}
\baselineskip 14pt

\maketitle
%
{\footnotesize
 \centerline{$^{\small 1}$ University of Tsukuba, Japan}
 \centerline{$^{\small 2}$  Aix-Marseille Universit\'{e}, France}
} 

\begin{abstract}
 We study the Rellich type theorem (RT) for the Maxwell operator $\hat H^D=\hat D\hat H_0$ on ${\bf Z}^3$
 in a constant anisotropic medium, i.e., the permittivity and permeability of which are constant
 non-scalar diagonal matrices. We also prove the unique continuation property (UCP) in the exterior
 of a compact convex set $\Kint\subset {\bf Z}^3$ for the perturbed Maxwell operator
 $\hat H^{D_p}=\hat {D}_p\hat H_0$ on ${\bf Z}^3$ for which the permittivity and permeability
 are locally perturbed from a constant matrix on a compact subset in $\Kint$.

\end{abstract}

\noindent{\em 2020 Mathematics Subject Classification.} Primary 35Q61, 47A40, 81Q35.\\
\noindent{\em Keywords}. Maxwell equations, Scattering theory of linear operators,
  Quantum mechanics on lattices.

\date{\today}

\section{Introduction}
\subsection{Presentation and Main results}
 Let us begin with general definitions.
 Consider an operator $\hat Q$ acting on the space $({\bf C}^m)^{{\bf Z}^d}$ of sequences 
 $(\hat u(n))_{n\in {\bf Z}^d}$ with values in ${\bf C}^m$, where $m,d\ge 1$.
 We efine the Besov type space
\[
 \calB_0^*({\bf Z}^d;{\bf C}^m) := \Big\{ \hat u: \; {\bf Z}^d \to {\bf C}^m \; \mid\;
   \lim_{R\to\infty} \frac1R\sum_{|n|<R} |\hat u(n)|^2 = 0\Big\}.
\]
 We say that the Rellich type theorem (RT) holds for $\hat Q$, if $\hat u\in\calB_0^*({\bf Z}^3;{\bf C}^6)$
 satisfies the equation
\begin{equation}
\label{e.Qu=0}
 \hat Q \hat u\, (n) = 0,\quad n\in {\bf Z}^d
\end{equation}
 outside a compact set in ${\bf Z}^d$, then $\hat u = 0$ on $\{n\in{\bf Z}^d$ $\mid$ $|n| > R\}$
 for some constant $R > 0$.

 (RT) is connected with the unique continuation property (UCP) in exterior domains of ${\bf Z}^d$. 
 Here, by an exterior domain we mean the set of the form $\Kext={\bf Z}^d\sauf \Kint$, 
 $\Kint$ being compact.
 We say that (UCP) holds for $\hat Q$ in $\Kext$ if, all $\hat u$ with compact support
 in ${\bf Z}^d$ satisfying $\hat Q\hat u = 0$ in $\Kext$ vanishes on $\Kext$.
 
 We consider (RT) and (UCP) for the discrete Maxwell operator $\hatHD$
 defined on the square lattice ${\bf Z}^3$.

 Let $\bfepsl$ and $\bfmu$ be the permittivity and the permeability for the ambient space ${\bf Z}^3$, 
 which  are the $3\times 3$ constant diagonal matrices with diagonal elements,
 $\epsl_1, \epsl_2, \epsl_3\in (0,\infty)$ for $\bfepsl$, and $\mu_1,\mu_2,\mu_3 \in (0,\infty)$
 for $\bfmu$.
 Letting $\hat H_0$ be the isotropic discrete Maxwell operator with
 $\bfepsl = \bfmu = I_{3\times 3}$, the anisotropic Maxwell operator is defined by
\[ \hatHD = \hat D \hat H_0, \]
 where we set
\[
 \hat D = \left(\begin{array}{cc} \bfepsl & 0_{3\times3} \\ 0_{3\times3} &\bfmu
  \end{array}\right).\qedhere
\]
 The perturbed Maxwell operator has the form $\hat H^{D_p}=\hat D_p \hat H_0$ where $\hat D_p$
 is a perturbation of $\hat D$, i.e., 
\[
 \hat D_p := \left(\begin{array}{cc} \tilde\bfepsl & 0_{3\times3} \\ 0_{3\times3} &\tilde\bfmu
  \end{array}\right),
\]
 $\tilde\bfepsl_j(n)>0$, $\tilde\bfmu_j(n)>0$ for $1\le j\le 3$ and $n\in{\bf Z}^3$, and
 $\hat D_p-\hat D$ has compact support.
 
 Our interest in (RT) and (UCP) appears in the application to inverse problems
 similar to those studied in \cite{AND.INV} for the Schr\"odinger operator on lattices.
 This aspect will be developed in further publications.
 
\medskip
 Following the methods  in \cite{ISO.REL}, \cite{AND.SPE2}, we pass to the Fourier series and consider
 $H^D=\hat D H_0$, which is a multiplication operator on the real torus $\T^3\approx (\R/(2\pi {\bf Z}))^3$.
 The operator $H^D$ is then represented by the analytically fibered self-adjoint operator
 $\T^3 \ni x\mapsto H^D(x)=\hat D H_0(x)$ where $H_0(x)$ is a $6\times 6$ real symmetric matrix
 depending on the variable $y=(y_1,y_2,y_3)$, $y_i = \sin x_i$ only (see \eqref{def.matH0}).
 The equation $(\hatHD -\lambda)\hat u\, (n) = 0$ outside a compact set is then equivalent to
\begin{equation}
\label{e.Hbu=f}
 (H^D -\lambda)u = f ,
\end{equation}
 where $f$ is a trigonometric polynomial of the variable $x\in \T^3$, and where \eqref{e.Hbu=f}
 has to be taken in the sense of distributions. 
 Using the correspondence $\hat u \in \calB_0^*({\bf Z}^3)\iff u\in \calB_0^*(\T^3)$, the problem
 is reduced to the Besov space in terms of the variable $x$.
 Multiplying \eqref{e.Hbu=f} by the cofactor matrix of $H^D(x)-\lambda$ leads to  the equation
\begin{equation}
\label{e.qu=g}
 q(\cdot;\lambda)u = g \quad {\rm on} \quad \T^3,
\end{equation}
 where $q(x;\lambda):=\det(H^D(x)-\lambda)$ and $g$ is a trigonometric polynomial.
 Let us recall the factorization of $q(x;\lambda)$ and the spectral analysis of $H^D(x)$
 (see \cite{POI.SPE}).
 The spectrum $ \sigma(H^D(x))$ of $H^D(x)$ depends on $z= (z_1,z_2,z_3)$, $z_i=\sin^2 x_i$,
 only, and $q(x;\lambda)$ is a polynomial of $z$ and $\lambda$, which we denote by $p(z;\lambda)$.
 We write the diagonal matrices $\bfepsl = (\epsl_1,\epsl_2,\epsl_3)$, $\bfmu = (\mu_1,\mu_2,\mu_3)$
 and put:
\begin{equation}
\label{def.beta-alpha}
\begin{aligned}
 \bfbeta &:=& \bfepsl \times \bfmu =(\beta_1,\beta_2,\beta_3)\\
 \alpha_i &:=& (\epsl_j  \mu_k + \epsl_k \mu_j)/2, \\
 \gamma_i &:=&  \epsl_j \epsl_k \mu_j\mu_k, 
\end{aligned}
\end{equation}
 where $(i, j, k)$ ranges over the cyclic permutations of $(1,2,3)$, and $\bfalpha:=(\alpha_1,\alpha_2,\alpha_3)$.
 We put
\[
\begin{aligned}
 {\bf B}_0 &:=& \{(0,0,0)\}, \\
 {\bf B}_{3} &:=& \{\bfbeta \in {\bf R}^3\; \mid \; \beta_1\beta_2\beta_3 \neq 0\},\\
 {\bf B}_{12} &:=& {\bf R}^3 \setminus ({\bf B}_0\cup {\bf B}_3).
\end{aligned}
\]

 We have
\begin{equation}
\label{val.pz}
 q(x;\lambda) = p(z;\lambda) = \lambda^2\left(\lambda^4 - 2 \lambda^2\Psi_0(z) + \Psi_0^2(z) - K_0(z)\right),
\end{equation}
 where we put
\begin{equation}
\label{def.Psi0}
\begin{aligned}
 K_0(z) &:=& \frac{1}{4}\sum_{i=1}^3(\beta_iz_i)^2 - \frac12\sum_{1\leq i<j\leq 3}\beta_i\beta_jz_iz_j,\\
 \Psi_0(z) &:=& \alpha\cdot z = \alpha_1z_1 + {\rm c.p.}. \hspace{2,3cm} 
\end{aligned}
\end{equation}
 Then, $p(z;\lambda)$ is rewritten as:
\[
 p(z;\lambda) = \lambda^2(\tau^+(z) - \lambda^2)(\tau^-(z) - \lambda^2),
\]
 where we put:
\begin{equation}
 \tau^{\pm}(z) := \Psi_0(z) \pm \sqrt{K_0(z)}.
\label{taupmzdefine}
\end{equation}
 In \cite{POI.SPE}, it is shown that $\tau^-(z)\ge 0$ for all $z\in [0,+\infty)^3$, and that
\[
 \sigma(H^D(x)) = \{-\sqrt{\tau^+(z)},-\sqrt{\tau^-(z)},0,\sqrt{\tau^-(z)},\sqrt{\tau^+(z)}\} , \quad x\in\T^3.
\]
 We put
\[
 \lambda_\pm = \max_{[0,1]^3} \sqrt{\tau^\pm(z)}.
\]
 Note that $\lambda_+ \ge \lambda_- >  0$, and we have $\lambda_+=\lambda_-$ iff
  $\bfbeta \in {\bf B}_{0}$.

 Our first main result is concerned with (RT) and is stated in terms of the parameter
 $\bfbeta=\bfepsl\times\bfmu$ and $\lambda_{\pm}$. 
\begin{theor}
\label{th.2}
 (RT) for $\hatHD-\lambda$ holds in each of the following cases:
\begin{enumerate}
\item $\bfbeta \in {\bf B}_0$ and $0 < |\lambda| < \lambda_+$, 
\item $\bfbeta \in {\bf B}_3$ and $0<|\lambda|<\lambda_+$,
\item $\bfbeta \in {\bf B}_{12}$ and $0<|\lambda|<\lambda_-$.
\end{enumerate}
\end{theor} 
 Since $\lambda_- < \lambda_+$ if $\bfbeta\in  {\bf B}_{12}$, the case (3)
 in Theorem \ref{th.2} is critical.
\begin{theor}
\label{th.3}
 Assume $\bfbeta\in {\bf B}_{12}$ and $|\lambda|>\lambda_-$. Then (RT) for $\hatHD-\lambda$ fails.
\end{theor}

\begin{remark}
\label{rem1.th3}
 Even in the situation  of Theorem \ref{th.3} the following fact holds.
 Let $\lambda\in (\lambda_-,\lambda_+)$, and  $u\in\calB_0^*(\T^3)$ be such that $(H^D -\lambda)u$
 is a trigonometric polynomial of $x\in \T^3$.
 Then $v(x):=(\tau^-(z)-\lambda^2)u(x)$ is a trigonometric polynomial of  $x\in \T^3$,
 and so, $u(x)=(\tau^-(z)-\lambda^2)^{-1}v(x) \in L^2(\T^3)$.
 We consequently obtain the following result.
\begin{prop}
\label{prop.1}
 Assume $\bfbeta\in {\bf B}_{12}$ and $|\lambda|>\lambda_-$. 
 Let $\hat u \in \calB_0^*({\bf Z}^3)$ such that $(\hat H^{D_p}-\lambda)\hat u =0$ and $\hat u\neq 0$.
 Then, $\hat u$ is an eigenvector of $\hatHD$ associated with the eigenvalue $\lambda$.
\end{prop}
\end{remark}
 Our second main result is about (UCP) and is stated as follows.
\begin{theor}
\label{th.1.3}
 Let $\Kext:={\bf Z}^3\sauf \Kint$ where $\Kint$ is a compact convex set of ${\bf Z}^3$.
 Assume $\lambda\neq 0$. Then (UCP) holds for the operator $\hat H^{D_p} - \lambda$ in $\Kext$.
\end{theor} 
\begin{remark}
\label{rem2.th3}
 The proof of Theorem \ref{th.1.3} shows that (UCP) holds in the exterior of a compact convex set
 for any scalar operator $\hat Q$ where $Q$ is the multiplication by a trigonometric polynomial
 $Q(x)$ of the form $\sum_{j=1}^3 c^+_j e^{iax_j} + c^-_j e^{-iax_j}  + \sum_{m\in{\bf Z}^3} c(m) e^{im\cdot x}$
 with $a\in \N^*$, $c_j^\pm\neq 0$, and $\sum_j |m_j|<a$.
\end{remark}

\begin{remark}
\label{rem3.th3}
 The results of \cite{POI.SPE} showed that $\sigma_p(\hatHD) = \{0\}$, and so, showed that (UCP)
 holds for the unperturbed operator $\hatHD-\lambda$ in the whole ${\bf Z}^3$ if $\lambda\neq 0$.
 Nevertheless (UCP) does not hold for $\hatHD$ in ${\bf Z}^3$ since the kernel of $\hatHD$ contains
 trigonometric polynomials. 
\end{remark}

\begin{coro}
\label{coro.1.4}
 Let $\lambda\neq 0$ and $\Kext$ be as in Theorem \ref{th.1.3}.
 Assume that (RT) for $\hatHD-\lambda$ holds.
 Let $\hat u\in \calB_0^*({\bf Z}^3)$ be a solution of $(\hat H^{D_p} -\lambda)\hat u=0$ in $\Kext$.
 Then $\hat u$ vanishes in $\Kext$.
\end{coro} 
\begin{remark}
 The proof of Theorem \ref{th.1.3} shows that Theorem \ref{th.1.3} and so Corollary \ref{coro.1.4} 
 extend to the case where the perturbation $\hat D_p- \hat D$ is any operator of multiplication
 by real or complex valued function with compact support such that, for all $n$,
 $\hat D_p(n)$ is a self-adjoint positive matrix.
 Nevertheless, it is difficult to extend Corollary \ref{coro.1.4} to the operator $(\hatHD + \hat V-\lambda)$,
 where $\hat V$ is an operator of multiplication on ${\bf Z}^3$ with compact support,
 since (UCP) fails for $\hat H_0$ and for $\hatHD$.
 In fact, the  trigonometric polynomial, $\phi^+(x):=(y,0_{{\bf R}^3})$ (we can choose also $\phi^-(x):=(0_{{\bf R}^3},y)$)
 $\in {\bf R}^6$ is an eigenmode for $H_0$ (and so for $H^D$)
 associated with the eigenvalue 0, then (UCP) obviously fails for $\hat H_0$ in the whole space ${\bf Z}^3$.
 We then have the following general negative result.\\
 \indent
 Let $\lambda\in\R$. Then there exists a compactly supported multiplication operator $\hat V$
 and a compactly supported eigenmode $\hat u$ of $\hat H_0+\hat V$ associated
 with  the eigenvalue $\lambda$, i.e.,
\[ (\hat H_0+\hat V-\lambda)\hat u=0.\]
\indent
 In fact, setting $u=\phi^+$ and $\hat V=\lambda \cdot 1_{\supp \hat u}$, the conclusion follows.\\
\end{remark}

\subsection{Plan of the paper}
 In \S\ref{sec.2} we give the definition of the discrete Maxwell operator, and recall its main spectral properties.
 In \S\ref{sec3} we prepare Theorem \ref{th.31} which is a very slight extension of results
 in \cite{ISO.REL}, \cite{AND.SPE2}. We then give a proof of Theorem \ref{th.31}.
 In \S\ref{sec.Fermi} we describe the Fermi surface $M^{\bf C}(q(\cdot;\lambda))$.
 The analysis begins with the description of the Fermi surface $M^{\bf C}(p(\cdot;\lambda))$
 which is simpler since the function $p(\cdot;\lambda)$ is a polynomial of second order
 of the variable $z$.  The key is Lemma \ref{l.topo1}.
 In \S\ref{sec.Proof-main} we give the proofs of Theorems \ref{th.2}, \ref{th.3}, \ref{th.1.3},
 Proposition \ref{prop.1} and Corollary\ref{coro.1.4}.

\subsection*{Notations}
 We use the following notations.
 Letting two sets $E,F$, we denote by $\calA(E;F)=F^E$ the set of applications from $E$ into $F$.
 For $n=(n_1,\ldots,n_d)\in{\bf Z}^d$ we set $|n|=\sum_{j=1}^d |n_j|$.
 For a map $f : {\bf R} \to {\bf R}$ (respect., a function from $\T^1$ into ${\bf R}^1$),
 we denote by $f$ again the mapping ${\bf R}^d$ (respectively, $\T^d$) $\ni x \mapsto (f(x_1),\ldots, f(x_d))\in {\bf R}^d$.
 If $E\subset\T^d$ or $E\subset{\bf R}^d$ we denote by $f(E)$ the set $\{f(x)$ $\mid$ $x\in E\}$.
 In particular we set, for $x\in \T_{\bf C}^d\cup {\bf C}^d$,
\begin{equation}
\label{y=sinxformula}
\begin{aligned}
  y &:= \sin x =(\sin x_1,\dots, \sin x_d), \\
  z &:= (\sin^2x_1, \dots, \sin^2x_d). 
\end{aligned}
\end{equation}
 We use the abreviated expressions "iff" which means "if and only if" and "$+$ {\rm c.p.}" in a computation
 to indicate that the value is the addition of the two other similar terms obtained by two successive
 permutations of $1\to 2\to 3\to 1$. For example, an expression written as $p_1+q_{23}+$ c.p.
 is the sum equal to $p_1+q_{23}+p_2+q_{31}+p_3+q_{12}$.
 For $\tilde E\subset \T_{\bf C}^d\cup{\bf C}^d$, we set $\sin^2(\tilde E)=\{z\in{\bf C}^d$ $\mid$ $x\in \tilde E\}$,
 and letting $E\subset {\bf C}^d$ we set $(\sin_{\bf C}^2)^{-1}(E)=\{x\in{\bf C}^d$ $\mid$ $z\in E\}$
 and $(\sin_\T^2)^{-1}(E)=\{x\in\T_{\bf C}^d$ $\mid$ $z\in E\}$.
 We denote by $\calC^\infty_c(E)$ the space of real $\calC^\infty$ functions defined in $E$
 with compact support in $E$.
 For $m\ge 1$ we denote by $(e_1,\ldots,e_m)$ the standard basis of ${\bf R}^m$ or of ${\bf C}^m$.
 The spaces ${\bf R}^m$, ${\bf C}^m$ are equipped with the inner product
\[
\begin{aligned}
 <f,g>_{{\bf R}^m}  & :=   \sum_{j=1}^m f_jg_j, \quad f=(f_j)_{1\le j \le m},\; g=(g_j)_{1\le j \le m}, \\
 <f,g>_{{\bf C}^m} & :=   \sum_{j=1}^m f_j\ove{g_j}, \quad f=(f_j)_{1\le j \le m},\; g=(g_j)_{1\le j \le m}.
\end{aligned}
\]
%
 (The formula for $<f,g>_{{\bf R}^m}$ extends to the case where $f_j$ and $g_j$ are complex valued.)
 If $E\subset {\bf R}^m$ (respect., $E\subset {\bf C}^m$) then we denote
 $E^\perp:=\{g\in {\bf R}^m$ $\mid$ $<f,g>_{{\bf R}^m}=0$, $\forall f\in E\}$ (respect.,
 $E^{\perp_{\bf C}}:=\{g\in {\bf C}^m$ $\mid$ $<f,g>_{{\bf C}^m}=0$, $\forall f\in E\}$).
 The Besov space $\calB_0^*(\T^d;{\bf C}^m)$,  simply denoted by $\calB_0^*(\T^d)$,
 is the space of tempered distributions $u\in\calS'(\T^d;{\bf C}^m)$ such that $\hat u\in \calB_0^*({\bf Z}^d)$.

 If $\calA$ is a $n\times p$ matrix with real coefficients we denote
 $\Ima\calA:=\calA({\bf R}^p) \subset {\bf R}^n$, $\Ima_{\bf C}\calA:=\calA({\bf C}^p)\subset{\bf C}^n$,
 $r= \rg(\calA):= {\rm dim}(\Ima\calA)$.

\section{The discrete-Maxwell Operator}
\label{sec.2}
\subsection{Description by Fourier series}
 Let ${\bf Z}^3=\{n=(n_1,n_2,n_3)$ $\mid$ $n_j\in {\bf Z}\}$, $\T^3\approx (\R/(2\pi {\bf Z}))^3$ and $U$ be
 the unitary transform between $L^2(\T^3)$ and $l^2({\bf Z}^3)$:
\[ (Uf)(n) := \hat f (n) = (2\pi)^{-\frac32} \int_{\T^3} e^{inx} f(x) \, \rmd x, \quad n\in{\bf Z}^3,\]
 so that any $f\in L^2(\T^3)$ can be written
\[
 f(x) = (U^* \hat f)(x) := (2\pi)^{-\frac32}\sum_{n\in{\bf Z}^3} e^{-inx} \hat f(n), \quad x\in\T^3.
\]
 The isotropic discrete Maxwell operator is the bounded operator $H_0$ on $\calH=(L^2(\T^3;{\bf C}))^6$
 defined by
\[ \hat H_0 = UH_0U^* , \]
 where $H_0$ is the operator of multiplication by the real symmetric $6\times 6$ matrix:
\begin{equation}
\label{def.matH0}
 H_0(x) := \left(\begin{array}{cc} 0_{3\times3} & \tilde M(y)\\ -\tilde M(y) & 0_{3\times3} \end{array}\right)
  \in {\bf R}^6,  \quad x\in\T^3,
\end{equation}
 where $y=\sin x$ (see \eqref{y=sinxformula}) and $\tilde M(y)$ is the real anti-symmetric $3\times 3$ matrix:
\[
 \tilde M(y) := \left(\begin{array}{ccc}
  0 & -y_3 & y_2\\ y_3 &0 & -y_1\\ -y_2 & y_1 & 0 \end{array}\right) , \quad y\in {\bf R}^3.
\]
 Then $H_0$ is a bounded self-adjoint operator on $\calH =  L^2(\T^3,\rmd x;{\bf C}^6)$ 
 with the inner product
\[ (u,v) :=  \int_{\T^3} <u(x),v(x)>_{{\bf C}^6} \,\rmd x  =  \int_{\T^3} \sum_{j=1}^6 u_j(x)\ove{v_j(x)}\, \rmd x .\]
 The anisotropic discrete-Maxwell operator is defined by
\[ \hatHD := \hat D \hat H_0, \]
  with
\[
 \hat D := \left(\begin{array}{cc} \bfepsl & 0_{3\times3} \\ 0_{3\times3} &\bfmu
  \end{array}\right).
\]
 We set $H^D :=U^*\hatHD U$ so, since $\hat D$ is constant,
\[
 H^D = U^*(\hat D \hat H_0)U = \hat D U^* \hat H_0 U = \hat D H_0.
\]
 The relation
\[
  (\hat D^{-1} H^D u,v)_{\calH}= (H_0 u,v)_{\calH}
\]
 shows that the operator $H^D$ is bounded and self-adjoint on the space $\calHD=\calH$
 equipped with the hilbertian product
\[ (u,v)_{\calHD} := (\hat D^{-1}u,v) = \int_{\T^3} <\hat D^{-1}u(x),v(x)>_{{\bf C}^6} \,\rmd x.\]
%
 We write
\[ H^D = \int_{\T^3}^{\oplus} H^D(x) \,\rmd x, \]
 where $H^D(x)$ is self-adjoint on ${\bf C}^6$ equipped with the inner product
\[
   <u(x),v(x)>_{{\bf C}^6,D} := <\hat D^{-1}u(x),v(x)>_{{\bf C}^6}.
\]

\subsection{Spectrum of $H^D(x)$}
 Since $H^D(x)$ depends only on the variable $y=\sin x$ we write $H^D(x)=h^D(y)$.
 In this subsection we consider a more general case $y\in {\bf R}^3$, and let $z_j=y_j^2$,
 $z\in [0,+\infty)^3$.
 We recall some results in \cite{POI.SPE} using the notations $\bfbeta$, $\bfalpha$ in \eqref{def.beta-alpha}.
 Since $\bfbeta\cdot \bfepsl=0$ and $\epsl_i>0$ for all $1 \leq i \leq 3$ then there exists
 $1 \leq j\leq 3$ such that $\beta_j \beta_i\le 0$ and $\beta_k \beta_i\ge 0$ for $i,k\neq j$.
 If two of the $\beta_j$'s vanish then $\bfbeta$ vanishes.
 Moreover $\bfbeta$ is replaced by $-\bfbeta$ if $\bfepsl$ and $\bfmu$ are exchanged,
 which involves the same analysis.
 Hence if $\bfbeta\neq0$ then we can assume without loss of generality that
 
 \medskip
 \noindent
 {\bf (A0)}: $\beta_1\ge \beta_2 > 0 >  \beta_3$ or $\beta_1> \beta_2 = 0 >  \beta_3$.

\medskip
 Let us observe that $\tau^{\pm}(z)$ defined by \eqref{taupmzdefine} satisfy
\[
   \tau^+(z)  \ge  \tau^-(z)  >0 \quad z\in [0,+\infty)^3\sauf\{0_{{\bf R}^3}\}.
\]
\begin{prop} (Spectrum of $h^D(y)$.)
\label{prop.spectreH}
 If $y=0_{{\bf R}^3}$ then $h^D(y)=0_{6\times 6}$.\\
 Let $y\in {\bf R}^3\sauf \{0_{{\bf R}^3}\}$. Then $0\in \sigma(h^D(y))$ has multiplicity two with eigenvectors
 $(y_1,y_2,y_3,0,0,0)= y\otimes 0_{{\bf C}^3}$ and $(0,0,0,y_1,y_2,y_3)= 0_{{\bf C}^3}\otimes y$.\\
 1) Assume $\bfbeta=0$ and $y\in {\bf R}^3 \sauf \{0_{{\bf R}^3}\}$.
 Then, $K_0\equiv 0$ and all the eigenvalues have multiplicity two.
 Moreover, the nonzero eigenvalues of $h^D(y)$ are
\[
 \pm\sqrt{\Psi_0(z)}=\pm \sqrt{\epsl_2\mu_3 z_1  + \epsl_3\mu_1 z_2+\epsl_1\mu_2 z_3},
\]
 where $z_j:=y_j^2$ hence $z\in [0,\infty)^3 \sauf \{0_{{\bf R}^3}\}$.

 2) Assume $\bfbeta\neq0$ (so (A0) holds) and $y\in {\bf R}^3 \sauf \{0_{{\bf R}^3}\}$.
 Then the nonzero eigenvalues of $h^D(y)$ are
\begin{itemize}
\item $\sqrt{\tau^+(z)}$ and $-\sqrt{\tau^+(z)}$, simple iff $K_0(z)\neq 0$,
\item $\sqrt{\tau^-(z)}$ and $-\sqrt{\tau^-(z)}$, simple iff $K_0(z)\neq 0$.
\item $\sqrt{\tau^+(z)}=\sqrt{\tau^-(z)}$ and  $-\sqrt{\tau^+(z)}=-\sqrt{\tau^-(z)}$, double iff $K_0(z)= 0$.
\end{itemize}
 If, in addition, $\beta_2=0$, then, $\tau^+$ and $\tau^-$ are linear with respect to $z$:
\[
\begin{aligned}
 \tau^+(z) &=&  \epsl_2\mu_3 z_1+ \epsl_3\mu_1 z_2 + \epsl_2\mu_1 z_3,\\
 \tau^-(z) &=& \epsl_3\mu_2 z_1+ \epsl_3\mu_1 z_2 + \epsl_1\mu_2 z_3.
\end{aligned}
\]
\end{prop}
 By Proposition \ref{prop.spectreH},  we observe that: If $(z_1,z_3)\neq 0_{{\bf R}^2}$,
 then the nonzero eigenvalues of $h^D(y)$ are $\pm\sqrt{\tau^+(z)}$, simple.
 If $(z_1,z_3)= 0_{{\bf R}^2}$ and $z_2\neq 0$, then the nonzero eigenvalues of $h^D(y)$ are
\[
 \pm\sqrt{\tau^+(z)}=\pm\sqrt{\tau^-(z)}=\pm \sqrt{\alpha_2} |y_2| =\pm \sqrt{\epsl_3\mu_1} |y_2| ,
\]

\subsection{Spectrum of $H^D$}
 We set
\[
 \lambda_\pm := \max_{[0,1]^3} \sqrt{\tau^\pm(z)} .
\]
 The following result is proved in \cite{POI.SPE}.
\begin{prop}
\label{p.HD-spectre}
(1)  The spectrum of $H^D$ is $\sigma(H^D) = [-\lambda_+,\lambda_+]$. \\
\noindent
(2) The pure point spectrum is $\sigma_{pp}(\hatHD)=\{0\}$ and
 the eigenvalue $0$ of $\hatHD$ has infinite multiplicity. \\
 \noindent
 (3) The singular continuous spectrum of $\hatHD$ is empty.
\end{prop}
\begin{remark}
 We have, in addition,

\smallskip
\noindent
(1) $\lambda_+ =\sqrt{\tau^+(1,1,1)}$, $\lambda_- = \max \{\sqrt{\tau^-(1,1,1)},\sqrt{\tau^-(1,1,0)}\}$. \\
\noindent
(2) If $\bfbeta\neq0$ then $\lambda_+ >\lambda_-$.
\end{remark}
%
%
\section{A general tool for the proof of Rellich's property}
\label{sec3}
 We apply to \eqref{e.qu=g} the following general result.
\begin{theor}
\label{th.31}
 Let $m\ge 1$, $Q$ a scalar trigonometric polynomial of  $x\in \T^d$, $d\ge 2$ satisfying
 $\ove{Q(x)}=Q(x)$, $x\in \T^d$.
 Consider the complex $d$-dimensional torus $\T^d_{\bf C}$, extend analytically $Q$
 to $\T^d_{\bf C}$, and set
\[
\begin{aligned}
 M^{\bf C}(Q) &:=& \{x\in \T^d_{\bf C} \; \mid \; Q(x) =0\}, \\
 M^{{\bf C}}_{sgn}(Q) &:=& \{x\in \T^d_{\bf C}\; \mid \; Q(x) =0, \; \nabla Q(x)=0\}, \\
 M^{{\bf C}}_{reg}(Q) &:=& M^{{\bf C}}(Q) \sauf M^{{\bf C}}_{sgn}(Q).
\end{aligned}
\]
 Assume\\
 {\bf (A-1-1)}$(Q)$: $M^{{\bf C}}_{sgn}(Q)$ ($\subset \T_{\bf C}^d \approx(\T\times{\bf R})^d$) has Hausdorff
 ($2d-1$)-measure zero,\\
 {\bf (A-1-2)}$(Q)$: $M^{{\bf C}}_{sgn}(Q)\cap \T^d$ is discrete,\\
 {\bf (A-2)}$(Q)$: Each connected component of $M^{{\bf C}}_{reg}(Q)$ intersects $\T^d$ and
 the intersection is a $(d-1)$-dimensional real analytic submanifold of $\T^d$.
 
 Let $u\in \calB_0^*(\T^d)$ be a solution of the equation
\begin{equation}
\label{e2.Qu=g}
 Qu = g \quad {\rm on} \quad \T^d ,
\end{equation}
 where $g$ is a trigonometric polynomial of  $x\in \T^d$ with values in ${\bf C}^m$.
 Then $u$ is a trigonometric polynomial.
\end{theor}
\begin{proof}
 Lemma 5.2 of \cite{AND.SPE2} shows that $u\in \calC^\infty(\T^d\sauf M^{{\bf C}}_{sgn}(Q))$
 and $g=0$ on $M^{{\bf C}}_{reg}(Q)\cap \T^d$.
 Lemma 5.3 of \cite{AND.SPE2} shows that $g=0$ on $M^{{\bf C}}_{reg}(Q)$. 
 Let us define the meromorphic function $v(x):=g(x)/p(x)$ for $x\in \T^d_{\bf C}\sauf M^{\bf C}(Q)$.
 The proof of \cite[Lemma 5.4]{AND.SPE2}) shows that $v$ extends continuously to $\T^d_{\bf C}$.
 In fact this proof is based on Assumption (A-1)'($Q$) (i.e., $M^{{\bf C}}_{sgn}(Q)$ is discrete)
 which is stronger than (A-1-1)-(A-1-2)($Q$).
 However \cite[Lemma 5.3]{AND.SPE2}, which does not use Assumption (A-1)'($Q$),  implies
 that $v$ is analytic near $M^{{\bf C}}_{reg}(Q)$ and the singularities of $v$ are localized
 in $M^{{\bf C}}_{sgn}(Q)$. Thanks to (A-1-1), the closed set $M^{{\bf C}}_{sgn}(Q)$ is negligible 
 in the sense of Shiffman in \cite{SHI.REM1} and \cite{SHI.REM2}, which means that $v$ extends analytically
 to $\T^d_{\bf C}$ (see also \cite{CAC.SUL} and \cite{STO.REM}).
 
 Let us set $u'=u-v|_{\T^d}$. Thus, $u'$ belongs to $\calB_0^*(\T^d)$ and satisfies
 $\supp u'\subset \T^d\cap M^{{\bf C}}_{sgn}(Q)$. Therefore, thanks to Assumption (A-1-2), 
 $\supp u'$ is discrete.
 Let us prove  $u'=0$.
 In fact, if $u'\neq 0$, then there exists $x^*\in \supp u'$.
 Let $\chi$ be a smooth cut-off function on $\T^d$ such that
 $\supp (\chi u')\subset \{x^*\}$ and $\chi(x^*)\neq 0$.
 We then use the following equivalent characterization of $\calB_0^*(\T^d)$ (see \cite{ISO.REL}).
 Take a $\calC^\infty$-partition of unity $\{\chi_l\}_{1\le l\le L}$ on $\T^d$ where the support
 of $\chi_l$ is sufficiently small. Then
\[ \calB_0^*(\T^d;{\bf C}^m) := \{w\in \calS'(\T^d;{\bf C}^m)\; \mid \; \chi_l w \in\calB_0^*({\bf R}^d;{\bf C}^m) \} ,\]
 where
\[
 \calB_0^*({\bf R}^d;{\bf C}^m) := \{w\in \calS'({\bf R}^d;{\bf C}^m)\; \mid \;
 \lim_{R\to +\infty} \frac1R \int_{|\xi|<R} |w(\xi)|^2 d\xi =0 \}.
\]
 Let $l$ be such that $\chi_l(x^*)\neq 0$.
 Thus $u'':=\chi_l\chi u'$ belongs to $\calB_0^*({\bf R}^d;{\bf C}^m)$ and is a finite linear combination
 of derivatives of Dirac distribution $\delta_{x^*}$:
\[ u''= \sum_{|\alpha|\le N} b_\alpha \partial_\alpha \delta_{x^*},\] 
 where $N\ge 0$ and $\{b_{\alpha}\}_{|\alpha|=N}\neq \{0\}$. 
 Let $ \tilde u''_N(\xi)$ be the Fourier transform of $u''_N:=$ $\sum_{|\alpha|= N} b_\alpha \partial^\alpha \delta_{x^*}$.
 It satisfies
\[
 \tilde u''_N(\xi) := (2\pi)^{-d/2}\sum_{|\alpha|= N} b_\alpha \xi^\alpha e^{i\xi x^*}, \quad \xi\in{\bf R}^d .
\]
 Putting $\xi=|\xi|\om$ with $\om\in {\mathbb S}^{d-1}$, we get
\begin{eqnarray*}
 \frac1R \int_{|\xi|<R}  |\tilde u''_N(\xi)|^2 \, \rmd \xi &=& 
 (2\pi)^{-d/2} \frac1R \int_{0}^R \Big(\int_{{\mathbb S}^{d-1}}
  |\sum_{|\alpha|= N} b_\alpha \om^\alpha |^2 \, \rmd\om \Big) t^{2N+d-1} \, \rmd t \\
 &=& (2\pi)^{-d/2} R^{2N+d-2} \int_{{\mathbb S}^{d-1}} |\sum_{|\alpha|= N} b_\alpha \om^\alpha |^2 \,\rmd\om \\
 &=& CR^{2N+d-2}
\end{eqnarray*}
 with $C>0$, since $\{b_\alpha\}_{|\alpha|=N}\neq \{0\}$.
 A similar calculation shows that
\[
 \frac1R \int_{|\xi|<R}  |\tilde u''(\xi)- \tilde u''_N(\xi)|^2 \, \rmd \xi =O(R^{2N+d-4}).
\] 
 We thus have 
\[
 \frac1R \int_{|\xi|<R}  |\tilde u''(\xi)|^2 \, \rmd \xi \approx R^{2N+d-2}.
\]
 It contradicts $u'' \in \calB_0^*({\bf R}^d;{\bf C}^m)$.
 Hence $u'=0$ and $u=v|_{\T^d}$ is a trigonometric polynomial. 
\end{proof}

\begin{remark}
 1) Theorem \ref{th.31} is a  slight improvement of results in \cite{AND.SPE2} and \cite{ISO.REL},
 where Assumption (A-1-1)-(A-1-2) is replaced by: (A-1)'($Q$): $M^{{\bf C}}_{sgn}(Q)$ is discrete.
 
 2) Theorem \ref{th.2} for the second case follows from Conditions (A-1-1)-(A-1-2)  
 by taking $Q$ to be $q(\cdot;\lambda)$ in Theorem \ref{th.31} but not from Condition
 (A-1)'($q(\cdot;\lambda)$) of \cite{AND.SPE2} and \cite{ISO.REL}, which is not fullfield.
\end{remark}
%
%

\section{Properties of the complex Fermi variety}
\label{sec.Fermi}
\subsection{Properties in the complex variable $z$}
 We consider $\lambda\in {\bf R}^*$. Thanks to \eqref{val.pz} we write
\begin{equation}
\label{form.pz}
 p(z;\lambda)/\lambda^2   =  Az\cdot z + \bfb\cdot z +c \quad z\in{\bf C}^3,
\end{equation}
 where we put
\begin{equation}
\label{def.A}
 A := \left( \begin{array}{ccc}
 \gamma_1 & \epsl_3\mu_3 \alpha_3 & \epsl_2\mu_2 \alpha_2 \\
  \epsl_3\mu_3 \alpha_3  & \gamma_2 & \epsl_1\mu_1 \alpha_1 \\
  \epsl_2\mu_2 \alpha_2  & \epsl_1\mu_1 \alpha_1 & \gamma_3
  \end{array} \right), 
 \quad \bfb := - 2\lambda^2 \bfalpha  , \quad  c := \lambda^4,
\end{equation}
 and $\gamma_i$ is defined by \eqref{def.beta-alpha}.
%
%


\medskip
 We state analytic properties of the variety defined by the polynomial $p(\cdot ;\lambda)$ in a slightly general form. 
 We consider a polynomial of the form $P(z)= \calA z \cdot z+\bfb\cdot z+c$,
 $z\in {\bf C}^d$, where $\calA$ is a real symmetric $d\times d$ matrix, $\bfb\in {\bf R}^d$, $c\in\R$.
 We set the complex analytic variety $M^{\bf C}(P) :=\{z\in{\bf C}^d$ $\mid$ $P(z)=0\}$,
 $M^{\bf C}_{sgn}(P) :=\{z\in M^{\bf C}(P)$ $\mid$ $\nabla P(z)=0\}$ the singular part and
 $M^{\bf C}_{reg}(P):=M^{\bf C}(P)\sauf M^{\bf C}_{sgn}(P)$ the regular part.
\begin{lemma}
\label{l.Az2BzC} 
(1)  Assume $\calA=0$, $\bfb\neq 0$.
 Then $M^{\bf C}_{sgn}(P) = \emptyset$ and $M^{\bf C}_{reg}(P)=M^{\bf C}(P)$ is a $(d-1)$-dimensional connected
 analytic submanifold of ${\bf C}^d$. 
 In addition,  $M^{\bf C}_{reg}(P)\cap{\bf R}^d$ is a $(d-1)$-dimensional real analytic submanifold of ${\bf R}^d$. \\
 \noindent
(2) Assume $\calA\neq0$. Then:\\
\noindent
(2-1) Assume $\bfb\not\in \Ima \calA$.
 Then, $M^{\bf C}_{sgn}(P)=\emptyset$, $M^{\bf C}(P)=M^{\bf C}_{reg}(P)$ is an analytic connected manifold
 of dimension $d-1$ and the intersection $M^{\bf C}_{reg}(P)\cap{\bf R}^d$ is a $(d-1)$-dimensional
 real analytic submanifold of ${\bf R}^d$.

\noindent
(2-2) Assume $\bfb \in \Ima \calA$. 
 Then there exists a unique vector $z^0\in\Ima_{\bf C} \calA$ such that $\calA z^0= -\bfb/2$.
 In addition, $z^0\in {\bf R}^d$. 
 Furthermore, setting $c^*:= \calA z^0 \cdot z^0 = -\bfb z^0/2$ and $c':=c-c^*$, 
 the following assertions (2-2-1) and (2-2-2) hold.\\
%
\noindent
(2-2-1) Assume $c \neq c^*$. 
  Then $M^{\bf C}_{sgn}(P)$ is void and $M^{\bf C}(P)=M^{\bf C}_{reg}(P)$. Moreover,
  \begin{enumerate}
   \item Assume $0<s=r$ and $c > c^*$, or $0=s<r$ and $c < c^*$.
   Then, $M^{\bf C}(P)$ does not intersect ${\bf R}^d$.
   \item Assume $0\le s<r$ and $c > c^*$, or $0<s\le r$ and $c < c^*$.\\
   - Assume $r\ge 2$. Then $M^{\bf C}(P)$ is a $(d-1)$-dimensional connected complex analytic submanifold
   of ${\bf C}^d$ and $M^{\bf C}(P)\cap {\bf R}^d$ is a $(d-1)$-dimensional real analytic submanifold of ${\bf R}^d$.\\
   - Assume $r=1$. Assume $s=0$ and $c > c^*$ or $s=1$ and $c < c^*$.
    Then $M^{\bf C}(P)$ has two connected components which are $(d-1)$-dimensional connected
    complex analytic submanifolds of ${\bf C}^d$ whose intersection with ${\bf R}^d$ are
    $(d-1)$-dimensional real analytic submanifolds of ${\bf R}^d$.
  \end{enumerate}
 (2-2-2)  Assume $c = c^*$.
 Then,  $M^{\bf C}_{sgn}(P)$ is a $(d-r)$-dimensional complex analytic submanifold of ${\bf C}^d$
 and the intersection $M^{\bf C}_{sgn}(P)\cap{\bf R}^d$ is a $(d-r)$-dimensional real analytic
 submanifold of ${\bf R}^d$. Moreover,\\
 - Assume $r=d$.
 Then $M^{\bf C}_{sgn}(P)=\{0_{{\bf C}^d}\}$ and $M^{\bf C}_{reg}(P)$ is a connected $(d-1)$-dimensional
 submanifold of ${\bf C}^d$.
  \begin{enumerate}
   \item
 Assume $s=0$ (i.e. $\calA <0$) or $s=r$ (i.e. $\calA >0$). Then $M^{\bf C}_{reg}(P)\cap{\bf R}^d=\emptyset$.
   \item
 Assume $1 \leq s \leq d-1$ (and  $r=d$). 
 Then $M^{\bf C}_{reg}(P)\cap{\bf R}^d$ is a $(d-1)$-dimensional real analytic submanifold of ${\bf R}^d$.
   \end{enumerate}
  - Assume $r=d-1$.
  \begin{enumerate}
   \item
 Assume $s=0$ (i.e. $\calA \le 0$) or $s=r$ (i.e. $\calA \ge0$).
 Then $M^{\bf C}_{reg}(P)\cap{\bf R}^d = \emptyset$.
   \item
 Assume $1 \leq s\leq d-1$ (and  $r=d-1$).
 If $d>3$ then $M^{\bf C}_{reg}(P)$ is a connected $(d-1)$-dimensional submanifold of ${\bf C}^d$
 and $M^{\bf C}_{reg}(P)\cap{\bf R}^d$ is a $(d-1)$-dimensional real analytic submanifold of ${\bf R}^d$.
 If $d=3$ then $M^{\bf C}_{reg}(P)$ admits two connected components: 
 $M^{\bf C}_{reg}(P)=M^{\bf C}(P^+)' \cup M^{\bf C}(P^-)'$, where each $M^{\bf C}(P^\pm)'$ is a two-dimensional
 submanifold of ${\bf C}^3$ and each $M^{\bf C}(P^\pm)'\cap{\bf R}^3$ is a two-dimensional real analytic
 submanifold of ${\bf R}^3$. 
   \end{enumerate}

\end{lemma}
\begin{proof}
%
(1) (Case $\calA=0$, $\bfb\neq0$.) This is obvious.\\
\noindent
(2) (Case $\calA\neq0$.)
 We observe that if it is not void, then $M^{\bf C}_{reg}(P)$ is an analytic manifold of dimension $d-1$.
 If $M^{\bf C}_{sgn}(P)\neq\emptyset$, then $M^{\bf C}_{sgn}(P)$ is a smooth manifold, since its singular part,
\[ M^{\bf C}_{sgn,sgn}(P)=\{z\in {\bf C}^d\; \mid \; P(z)=0, \; \nabla P(z)=0, \; D^2 P(z)=0\} \]
 is necessarily void (we have $z\in M^{\bf C}_{sgn,sgn}(P)$ iff $\calA=0$, $b=0$, $c=0$, which is forbidden).
 Moreover, $M^{\bf C}_{sgn}(P)$ is characterized by the equations $\calA z=-\bfb/2,\quad  \bfb\cdot z/2+c=0$.
 Let $(s,r-s)$ be the signature of the quadratic form  $z\mapsto \calA z\cdot z$. 
 Gauss's reduction provides a {\em real} linear change of coordinates on $z$, so we can assume that
 $\calA$ is diagonal and $\calA z \cdot z=\sum_{j=1}^{s} z_j^2 - \sum_{j=s+1}^{r} z_j^2$.
%
%
%
%
 The equation of $M^{\bf C}(P)$ is then
\begin{equation}
\label{e.pz'=0}
 M^{\bf C}(P): \quad \sum_{j=1}^d a_j z_j^2 + \sum_{j=1}^d b_j z_j + c=0,
\end{equation} 
 with $a_j=1$ if $j\le s$, $a_j=-1$ if $s<j\le r$, $a_j=0$ if $j>r$.
 The equations of $M^{\bf C}_{sgn}$ are:
\[
 M^{\bf C}_{sgn}(P):\quad \sum_{j=1}^d a_j z_j^2 + \sum_{j=1}^d b_j z_j + c=0,
 \quad {\rm and} \quad 2a_j z_j+ b_j=0 \quad \forall j .
\]
%
%
(2-1) (Case $\bfb\not\in \Ima \calA$.) Hence, $r<d$ and there exists $j_0>r$ such that
 $b_{j_0}\neq 0$. 
 Thus $M^{\bf C}_{sgn}(P)=\emptyset$ and $M^{\bf C}(P)=M^{\bf C}_{reg}(P)$ is an analytic
 manifold of dimension $d-1$.
 Moreover the equation of $M^{\bf C}(P)$ becomes
\[  z_{j_0} = \frac{-1}{b_{j_0}}(\sum_{j=1}^r z_j^2 + \sum_{j=r+1, \; j\neq j_0}^d  b_j z_j + c), \]
 which shows that $M^{\bf C}(P)$ is connected and that $M^{\bf C}(P)\cap{\bf R}^d$ has dimension $d-1$.
 This proves the assertion.

\noindent
(2-2) (Case $\bfb \in \Ima \calA$.) 
 We denote $\Ima_{\bf C} \calA=\{\calA z$ $\mid$ $z\in {\bf C}^d\}$, $\ker_{\bf C} \calA=\{z\in {\bf C}^d$ $\mid$ $\calA z=0\}$
 and $\ker\calA=\ker_{\bf C}\calA \cap{\bf R}^d$. 
 We recall that, since $\calA$ is real symmetric, we then have
 $\Ima\calA:=\{\calA z$ $\mid$ $z\in{\bf R}^d\}=(\ker \calA)^\perp$, $\Ima_{\bf C}\calA=(\ker_{\bf C} \calA)^{\perp_{\bf C}}$,
 and
\[ {\bf R}^d=\Ima \calA \oplus \ker \calA  , \quad {\bf C}^d =\Ima_{\bf C} \calA \oplus \ker_{\bf C} \calA. \]
 Hence, $\calA$ is an automorphism on $\Ima_{\bf C} \calA$. It implies the existence and uniqueness
 of a vector $z^0\in\Ima_{\bf C} \calA$ such that $\calA z^0= -\bfb/2$. Since $\bfb$ is a real
 vector so is $z^0$.
 Set $c^*:= \calA z^0 \cdot z^0 = -\bfb z^0/2$, $c':=c-c^*$, and consider the translation $z':=z-z^0$.
 Then \eqref{e.pz'=0} becomes
\[ M^{\bf C}(P):\quad \sum_{j=1}^{s} {z'_j}^2 - \sum_{j=s+1}^{r} {z'_j}^2 + c'=0,\]
 and the equations of $M^{\bf C}_{sgn}(P)$ become
\[
  M^{\bf C}_{sgn}(P):\quad \sum_{j=1}^{s} {z'_j}^2 - \sum_{j=s+1}^{r} {z'_j}^2 + c'=0 \quad {\rm and} \ \ 
   z'_j=0  \ \  {\rm for}\ \  1\le j\le r,
\]
 i.e.,
\[
  M^{\bf C}_{sgn}(P):\quad c'=0 \quad {\rm and} \quad  z'_j=0  \ \  {\rm for}\ \ 1\le j\le r.
\]
%
(2-2-1) (Case $c \neq c^*$, i.e., $c'\neq0$.) We assume $c > c^*$, the other case being similar.
  Then, obviously, $M^{\bf C}_{sgn}(P)$ is void and $M^{\bf C}(P)=M^{\bf C}_{reg}(P)$.
  \begin{enumerate}
   \item  Assume $0<s=r$, i.e., $\calA\ge 0$. Obviously, $M^{\bf C}(P)$ does not intersect ${\bf R}^d$.
   \item  Assume $0\le s<r$.\\
   - Assume $r\ge 2$. 
  Since $M^{\bf C}(P)$ is an irreducible and regular algebric variety then it is $(d-1)$-dimensional and connected.
  In addition $M^{\bf C}(P)$ intersects ${\bf R}^d$; it contains the points $z\in{\bf R}^d$ such that
  $z'_r= \pm \sqrt{\sum_{j=1}^{s} {z'_j}^2 - \sum_{j=s+1}^{r-1} {z'_j}^2+c'}$
 and $\sqrt{\sum_{j=1}^{r-1} |z'_j|^2} < c'$.
 Hence $M^{\bf C}(P)$ is a $(d-1)$-dimensional real analytic submanifold of ${\bf R}^d$.\\
   - Assume $s=0$, $r=1$. Then
\[
 M^{\bf C}(P)  =  M^{\bf C}_+ \cup M^{\bf C}_-, \quad M^{\bf C}_\pm : =  \{z\in{\bf C}^d \; \mid \;  z'_1 = \pm \sqrt{c'} \}.
\]
 Clearly, $M^{\bf C}_+$ and $M^{\bf C}_-$ are the connected composants of $M^{\bf C}(P)$ and
 their intersection with ${\bf R}^d$ is $(d-1)$-dimensional.
  \end{enumerate}
(2-2-2) (Case $c = c^*$, i.e., $c'= 0$.)
 Then,  $M^{\bf C}_{sgn}(P)=\{z\in{\bf C}^d$ $\mid$ $z'_1=\ldots=z'_r=0\}$ and
 $M^{\bf C}_{sgn}(P)\cap{\bf R}^d=\{z\in{\bf R}^d$ $\mid$ $z'_1=\ldots=z'_r=0\}$ are affine spaces with same dimension $d-r$.\\
 - Assume $r=d$. Then $M^{\bf C}_{sgn}(P)=\{0_{{\bf C}^d}\}$.
 Since $d\ge 3$ then $M^{\bf C}_{reg}(P)$ is a irreducible and regular algebric variety and so
 it is a $(d-1)$-dimensional connected submanifold of ${\bf C}^d$.
 Clearly, $M^{\bf C}_{reg}(P)\cap{\bf R}^d$ is a regular subset of ${\bf R}^d$, but may be void.
  \begin{enumerate}
   \item
 Assume $s=0$ (i.e., $\calA <0$) or $s=r$ (i.e., $\calA >0$). Then $M^{\bf C}_{reg}(P)\cap{\bf R}^d=\emptyset$.
   \item
 Assume $1\le s\le  d-1$ (and  $r=d$). Then $M^{\bf C}_{reg}(P)$ intersects ${\bf R}^d$ since it contains
 the points $z\in{\bf R}^d$ such that $z'_d= \pm \sqrt{\sum_{j=1}^{s} {z'_j}^2 - \sum_{j=s+1}^{d-1} {z'_j}^2}$
 and $|z'_1| > \sqrt{\sum_{j=s+1}^{d-1} {z'_j}^2}$.
 Hence $M^{\bf C}_{reg}(P)\cap{\bf R}^d$ is a $(d-1)$-dimensional real analytic submanifold of ${\bf R}^d$.
   \end{enumerate}
 - Assume $r=d-1$.
  \begin{enumerate}
   \item
 Case $s=0$ or $s=r$ (i.e., $\calA \ge0$): the result is obvious.
   \item
 Assume $1\le s\le  d-1$ (and  $r=d-1$).\\
 -- Assume $d>3$. The proof is similar to the proof of case (ii)(B)($r=d$).\\
 -- Assume $d=3$ so $r=2$, $s=1$.
 We have $M^{\bf C}_{reg}(P)=M^{\bf C}_+(P)'\cup M^{\bf C}_-(P)'$ with
 $M^{\bf C}(P^\pm)' : = M^{\bf C}(P^\pm) \sauf M^{\bf C}_{sgn}(P)$,
 $M^{\bf C}(P^\pm) : =  \{z\in{\bf C}^3$ $\mid$ $z'_1 = \pm z'_2 \}$.
 Hence $M^{\bf C}(P^\pm)'  = \{z\in{\bf C}^3$ $\mid$ $z'_1 = \pm z'_2$ and $z'_2\neq 0\}$.
 Clearly, $M^{\bf C}(P^+)'$ and $M^{\bf C}(P^-)'$ are the connected composants of 
 the complex variety $M^{\bf C}(P)$ and each intersection $M^{\bf C}(P^\pm)'$ with ${\bf R}^3$ is
 a two-dimensional open set of the plane $\{z\in{\bf R}^3$ $\mid$ $z'_1 = \pm z'_2\}$.
   \end{enumerate}
 We have thus completed the proof of Lemma \ref{l.Az2BzC}.
\end{proof}
 We put
\begin{equation}
\label{def.g}
 \bfg :=(\epsl_1\mu_1\beta_1,\epsl_2\mu_2\beta_2,\epsl_3\mu_3\beta_3).
\end{equation}
\begin{lemma}
 We have the following relation:
\begin{equation}
\label{r.alphaImA}
 \bfalpha\cdot \bfg = \sum_{i=1}^3\epsl_i\mu_i\alpha_i\beta_i =  \frac12 \beta_1\beta_2\beta_3.
\end{equation}
\end{lemma}
\begin{proof}
 We easily observe that the following relations and their cyclic changes hold:
\begin{equation}
\label{r.alphabeta}
\begin{aligned}
 \alpha_1^2-\gamma_1 &=& \frac14 \beta_1^2,\\
 \epsl_1\mu_1 \alpha_1 - \alpha_2\alpha_3 &=& \frac14\beta_2\beta_3,\\
 \alpha_3\beta_2 + \alpha_2 \beta_3 &=& -\epsl_1\mu_1 \beta_1.
\end{aligned}
\end{equation}
 Then, we obtain
\begin{eqnarray*}
 \bfalpha\cdot \bfg  &=&
  \sum_{i=1}^3\epsl_i\mu_i\alpha_i\beta_i = -(\alpha_3\beta_2 + \alpha_2 \beta_3)\alpha_1
 + \epsl_2\mu_2\alpha_2\beta_2 + \epsl_3\mu_3\alpha_3\beta_3 \\
 &=& (\epsl_2\mu_2\alpha_2-\alpha_1\alpha_3)\beta_2 + (\epsl_3\mu_3\alpha_3-\alpha_1\alpha_2) \beta_3
 = \frac14 \beta_1\beta_3 \beta_2 + \frac14\beta_1\beta_2\beta_3 \\
 &=&  \frac12\beta_1\beta_2\beta_3.
\end{eqnarray*}
%
\end{proof}

\begin{lemma}
\label{l.rgA} 
 Let $A$ be defined by \eqref{def.A} and $\bfg$ by \eqref{def.g}.

\noindent
(1) Let $\{a_1,a_2,a_3\}$ be the spectrum of $A$. 
 Then, we have
\begin{equation}
\label{r.vpA}
 \det A=a_1a_2a_3=0, \quad  {\rm tr} A =  a_1+a_2+a_3= \gamma_1+\gamma_2+\gamma_3>0,
 \quad\Pi_{i\neq j} a_ia_j = -\bfg^2/4 \le 0.
\end{equation}
\noindent
(2) We have $\bfalpha\in \Ima A$ iff $\beta_1\beta_2\beta_3=0$.

\noindent
(3) Assume $\bfbeta=0$.
 Then $A=\bfalpha\bfalpha^T$, $\rg(A) = 1$,
 $\Ima A=\span(\bfalpha)$ and $\ker A$
 is the complex plane orthogonal to $\bfalpha$.

\noindent
(4) Assume $\bfbeta\neq 0$. Then $\rg(A) = 2$ and $\ker A= \span(\bfg)$.
 Moreover, $A$ has one positive eigenvalue and one negative eigenvalue.

\noindent
(5) Assume $\beta_1\beta_2\beta_3=0$ (and $\bfbeta\neq0$), so we assume (A0) with
 $\beta_1>0=\beta_2>\beta_3$.
 Then there exists a unique vector $z^*\in\Ima A$ such that $Az^*=\bfalpha$.
 We have $z^* = (0,1/\sqrt{\gamma_2},0)$ and
\begin{equation}
\label{c.alphaz*}
 \bfalpha\cdot z^* = 1.
\end{equation}
\end{lemma}
\begin{proof}
 (1) A simple computation shows that $A\cdot \bfg=0$ and, if $\bfbeta=0$, $A\cdot (1,1,1)=0$.
 Hence $\det A=0$. Moreover, we have
\begin{eqnarray*}
 \Pi_{i\neq j} a_ia_j &=& \gamma_2\gamma_3 - \epsl_1^2\mu_1^2\alpha_1^2 + {\rm c.p.}
  =  \epsl_1^2\mu_1^2(\epsl_2\epsl_3\mu_2\mu_3-\alpha_1^2) + {\rm c.p.} \\
 &=&   \epsl_1^2\mu_1^2(\epsl_2\epsl_3\mu_2\mu_3- (\epsl_2\mu_3+\epsl_3\mu_2)^2/4) + {\rm c.p.} \\
 &=& -\epsl_1^2\mu_1^2 \beta_1^2/4 + {\rm c.p.} = -\bfg^2/4 \le 0.
\end{eqnarray*}
\\
 (3) From \eqref{def.Psi0}, \eqref{val.pz} and \eqref{form.pz} we get $A=\bfalpha\bfalpha^T$.
 
 \noindent
 (2)
 Assume $\bfbeta = 0$. Since $A=\bfalpha\bfalpha^T$ (see (3)), then $\bfalpha\in\Ima A$.\\
 Assume $\bfbeta\neq 0$.
 Since $A$ is real symmetric then $\Ima A=\ker A^T=\bfg^\perp$, so, thanks to \eqref{r.alphaImA},
 $\bfalpha\in \Ima A\iff \bfalpha\cdot\bfg=0\iff\Pi_{j=1}^3\beta_j=0$.

\noindent
 (4) 
 (Case $\bfbeta\neq 0$.) Thanks to \eqref{r.vpA} and since $\bfg\neq 0$ then
 $\rg(A)=2$ and $\ker A=\span(\bfg)$, and (2) follows from \eqref{r.vpA}.

\noindent
(5) The proof of Lemma \ref{l.Az2BzC} (see case (2-2)) with $(\calA,\bfb,d,r,s)=(A,-2\bfalpha,3,2,1)$
 shows the existence and uniqueness of such a $z^*$.
 Put $z = (0,1/\alpha_2,0)$ and observe that $\alpha_2=\sqrt{\gamma_2}$.
 Since $\bfg=(\epsl_1\mu_1\beta_1,0,\epsl_3\mu_3\beta_3)$, then $z\perp \bfg$ and so $z\in \Ima A$.
 In addition, we have $(Az)^T= (\epsl_3\mu_3\alpha_3/\alpha_2, \alpha_2, \epsl_1\mu_1\alpha_1/\alpha_2)$.
 Thanks to \eqref{r.alphabeta}, we obtain $Az=\bfalpha$. This shows that $z^*=z$.
 Finally, \eqref{c.alphaz*} is obvious and (5) is proved.
\end{proof}
%
 Let $\lambda\in{\bf R}^*$. We use the notations of Lemma \ref{l.Az2BzC}
 and apply Lemma \ref{l.Az2BzC} to $P:=\frac1{\lambda^2}p(\cdot;\lambda)$,
 so $(\calA,\bfb,c,d)=(A,-2\lambda^2\bfalpha,\lambda^4,3)$.
 In view of Lemmas \ref{l.Az2BzC} and \ref{l.rgA} we obtain
\begin{lemma}
\label{l.3} 
(1) Assume $\bfbeta=0$.
 Then $b\in\Ima A$, $z^0= \frac{\lambda^2}{|\bfalpha|^2}\bfalpha$, $c^*=c$, and
 $M^{\bf C}_{sgn}(P)$ is a two-dimensional complex submanifold of ${\bf C}^3$.
 
\noindent
(2)  Assume $\bfbeta\neq0$. \\
\noindent
(2-1) Assume $\Pi_{j=1}^3 \beta_j\neq 0$. 
 Then $\bfb\not\in \Ima \calA$, $M^{\bf C}_{sgn}(P)=\emptyset$, $M^{\bf C}(P)=M^{\bf C}_{reg}(P)$ is
 an analytic connected manifold of dimension two and its intersection with ${\bf R}^3$ is a real
 analytic manifold of dimension two.

\noindent
(2-2) Assume $\Pi_{j=1}^3 \beta_j= 0$. Then $\bfb\in \Ima \calA$, $z^0=\lambda^2 z^*$,
 $c^*=c$, $M^{\bf C}_{sgn}(P)$ is a straight line of ${\bf C}^3$,
 and the intersection $M^{\bf C}_{sgn}(P)\cap{\bf R}^3$ is a straight line of ${\bf R}^3$.
 In addition, $M^{\bf C}_{reg}(P)$ has two connected components, $M^{\bf C}(P^+)'$ and $M^{\bf C}(P^-)'$,
 and each $M^{\bf C}(P^\pm)'\cap{\bf R}^3$ is a two-dimensional real analytic submanifold of ${\bf R}^3$.
\end{lemma}
\begin{remark}
\label{rem.5}
 Let us supplement (5) of Lemma \ref{l.rgA} and  (2-2) of Lemma \ref{l.3}.
 We assume $\beta_1>0=\beta_2>\beta_3$.
 Then each $\tau^\pm$ is linear and we have
\[ A(z-z^0)(z-z^0)=\lambda^{-2} p(z;\lambda)=(\tau^+ -\lambda^2)(\tau^- -\lambda^2),\]
 where $z^0=\lambda^2 z^*$ is the corresponding point of Lemmas \ref{l.Az2BzC} and \ref{l.rgA}.
 Set $P=p(\cdot;\lambda)$ and the linear functions $P^\pm:=\tau^\pm-\lambda^2$.
 Then each $M^{\bf C}_{sgn}(P^\pm)$ is void and, since $P=\lambda^2 P^+ P^-$, we have
\begin{equation}
\label{val.MCsgnP}
 M^{\bf C}_{sgn}(P)=M^{\bf C}(P^+)\cap M^{\bf C}(P^-) =\{z\in {\bf C}^3 \; \mid \; K_0(z)=\Psi_0(z)-\lambda^2=0\}
\end{equation}
 is a complex analytic variety of dimension one.
 Then, 
\[
 M^{\bf C}_{reg}(P)= M^{\bf C}(P^+)'\cup M^{\bf C}(P^-)' ,
\]
 where $M^{\bf C}(P^\pm)':=M^{\bf C}(P^\pm)\sauf M^{\bf C}_{sgn}(P)$, as in the proof of Lemma \ref{l.Az2BzC}.
 Thanks to Lemma \ref{l.Az2BzC}, $M^{\bf C}(P^+)'$ and $M^{\bf C}(P^-)'$ are the two connected component
 of $M^{\bf C}_{reg}(P)$; they are two-dimensional submanifolds of ${\bf C}^3$.
\end{remark}
 Lemma \ref{l.3} shows that we can't apply Theorem \ref{th.31} directly to prove Rellich's properties
 if $\bfbeta=0$ since Assumptions (A-1-1)($p(;\lambda)$) fails.
\begin{lemma}
\label{l.44}
 Assume $\bfbeta=0$. Set $P=\Psi_0-\lambda^2$.
 Then $M^{\bf C}_{reg}(P)=M^{\bf C}(P)$ and $M^{\bf C}(P)\cap [0,1]^3$ is a closed convex set.
 In addition, we have\\
 \noindent
 (1) Let $|\lambda|\in (0,\lambda_+)$. Then $M^{\bf C}_{reg}(P)\cap (0,1)^3$ has dimension two.\\
 \noindent
 (2) Let $|\lambda|=\lambda_+$. Then $M^{\bf C}_{reg}(P)\cap [0,1]^3=\{(1,1,1)\}$.
\end{lemma}
\begin{proof}
 Obviously, $M^{\bf C}_{sgn}(P)=\emptyset$ and $M^{\bf C}(P)\cap [0,1]^3$ is
 a closed convex set as intersection of  the plane $\bfalpha\cdot z=\lambda^2$
 of ${\bf R}^3$ and the cube $[0,1]^3$.\\
 \noindent
 (1)
 Let $|\lambda|\in (0,\lambda_+)$. 
 Set $z_\lambda=\frac{|\lambda|}{\lambda_+}(1,1,1)\in (0,1)^3$, so 
 $\Psi_0(z_\lambda)=\tau^\pm(z_\lambda)=\lambda^2$, since we have
 $\lambda_+=\Psi_0(1,1,1)=\sum_{j=1}^3\alpha_j$.
 Thus, $z_\lambda \in M^{\bf C}(P)\cap (0,1)^3$, and the conclusion follows.\\
\noindent
(2)  Assume $|\lambda|=\lambda_+$. Since $\partial_{z_j}\Psi_0(z)=\alpha_j>0$ for all $j$
 then $|\Psi_0(z)|<\lambda_+ = \Psi_0(1,1,1)$ for all $z\in [0,1]^3\sauf\{(1,1,1)\}$.
 Hence $\{(1,1,1)\}=M^{\bf C}(P)\cap [0,1]^3$, which proves (2).
\end{proof}
\begin{lemma}
\label{l.45}
 Under the assumptions in Remark \ref{rem.5} with the same notations (hence we assume
 $\bfbeta\neq 0$ and $\Pi_{j=1}^3\beta_j=0$ and we have $P=p(\cdot;\lambda)$, 
 $P^\pm:=\tau^\pm-\lambda^2$), we have\\
 (1) The intersection $M^{\bf C}_{sgn}(P)\cap[0,1]^3$ has one point at most.\\
 (2) Let $|\lambda| \in (0,\lambda_-)$. Each $M^{\bf C}(P^\pm)'\cap (0,1)^3$ has dimension two.\\
 (3-1) Let $|\lambda| \in (0,\lambda_+)$. Then $M^{\bf C}(P^+)'\cap (0,1)^3$ has dimension two.\\
 (3-2) Let $|\lambda| \in (0,\lambda_-)$. Then $M^{\bf C}(P^-)'\cap (0,1)^3$ has dimension two.\\
 (3-3) Let $|\lambda| \ge \lambda_-$. Then  $M^{\bf C}(P^-)'\cap [0,1]^3$ has at most one point.
\end{lemma}
\begin{proof}
1)
 Thanks to \eqref{val.MCsgnP} and since $\beta_1>0>\beta_2$,  we have the equivalence
 ($K_0(z)=0$ and $z\in [0,1]^3$) iff $z_1=z_3=0$. Hence
\[
 M^{\bf C}_{sgn}(P) \cap [0,1]^3  =\{z=(0,z_2,0)\; \mid \; z_2\in [0,1] \quad {\rm and}
  \quad z_2=\lambda^2/\alpha_2\} \subset \{z^0\},
\]
 with $z^0=(0,\lambda^2/\alpha_2,0)$. This proves (1).\\
 (2), (3-1) and (3-2).  It is similar to the proof of Lemma \ref{l.44}.\\
 (3-3) If $|\lambda|>\lambda_- = \sup_{[0,1]^3} \tau^-$ then, obviously, 
 $M^{\bf C}(P^-) \cap \T^3$ is void.
 If $|\lambda|=\lambda_- = \sqrt{\tau^-(z)}$ and $z\in [0,1]^3$, then $z=(1,1,1)$,
 so $M^{\bf C}(P^-) \cap [0,1]^3 = \{(1,1,1)\}$. The conclusion follows.
\end{proof}
\begin{lemma}
\label{l.46}
 Assume $\Pi_{j=1}^3\beta_j\neq0$.
 Let $\lambda\neq 0$ and set $P=\lambda^{-2} p(\cdot;\lambda)$, so $M^{\bf C}_{sgn}(P)=\emptyset$.\\
 (1) Let $|\lambda| \in (0,\lambda_+)$. Then $M^{\bf C}(P)\cap (0,1)^3$
 has dimension two.\\
 (2) Let $|\lambda|=\lambda_+$. Then $M^{\bf C}(P)\cap [0,1]^3=\{(1,1,1)\}$.
\end{lemma}
\begin{proof}
 We have $P(z)=(\tau^+(z)-\lambda^2)(\tau^-(z) - \lambda^2)$ so
 $M^{\bf C}(P)\cap (0,1)^3=\{z\in (0,1)^3$ $\mid$ $\tau^+(z)=\lambda^2\}\cup \{z\in (0,1)^3$ $\mid$ $\tau^-(z)=\lambda^2\}$.
 Since $\nabla P(z)\neq 0$ for all $z\in [0,1]^3$ then $M^{\bf C}(P)\cap (0,1)^3$
 has dimension is 2 iff it is non empty.\\
(1) Let $|\lambda| \in (0,\lambda_+)$ and set $z(t)=t(1,1,1)$.
 Since $\tau^+(z(0))=0$ and $\tau^+(z(1))=\lambda_+^2$ then there exists $t\in (0,1)$
 such that $\tau^+(z(t))=\lambda^2$. Hence $z(t)\in M^{\bf C}(P)\cap (0,1)^3$.
(2)  Let $|\lambda|\ge \lambda_+$. Since $\lambda_+<\lambda_-$ then $\tau-(z)<\lambda^2$
 for all $z\in [0,1]^3$.
 Assume $|\lambda|> \lambda_+$. Then $\tau^+(z)<\lambda^2$ for all $z\in [0,1]^3$
 so $M^{\bf C}(P)=\emptyset$. 
 Assume $|\lambda|= \lambda_+$. Since we have
 ($\tau^+(z)= \lambda_+^2$ and $z\in [0,1]^3$) iff $z=(1,1,1)$, the conclusion then follows.
\end{proof}

\subsection{The complex Fermi variety in the $x$-variable} 
 We denote $\sin_\T:\: \T_{\bf C}^d\ni x \mapsto z=\sin^2 x \in {\bf C}^d$ and set
\[ X_{0,1} := (\sin_\T^2)^{-1} (\{0,1\}^3) =\{x\in \T^3 \; \mid \: z\in \{0,1\}^3\}.\]
\begin{lemma}
\label{l.topo1}
 Let $d\ge r\ge 1$ and  $E\subset{\bf C}^d$ a connected set.
 Assume that $\tilde E:=(\sin_\T^2)^{-1}(E)$ is a $r$-dimensional smooth submanifold of $\T_{\bf C}^d$.
 Let $C_x$ a connected component of $\tilde E$. Then $C_x$ is open and closed and
 $\sin^2(C_x)=E$.
\end{lemma}
\begin{proof}
 Let us consider the topological set $E$ with the topology induced by those of ${\bf C}^d$
 and the topological set $\tilde E$ with the topology induced by those of $\T_{\bf C}^d$.
 Then, it is easy to see that $\phi_E$: $\tilde E \ni x\mapsto z \in E$ is continuous. 
 Since $\tilde E$ is a $r$-dimensional smooth submanifold of $\T_{\bf C}^d$, near any $x^0\in \tilde E$,
 there exists an open ball $B(x^0,r)\subset \T_{\bf C}^d$ and a smooth function $f:B(x^0,r)\mapsto {\bf C}$
 such that $\nabla f(x)\neq 0$ for all $x\in B(x^0,r)$.  Then, $\tilde E\cap B(x^0,r)= f^{-1}(0_{{\bf C}})$.
 Hence $\tilde E$ is locally connected and each of its connected components is open and closed.
 We now prove that the map $\phi_E$ is both open and closed.
 It then implies that the set $\phi_E(C_x)$ is open and closed hence equals to $E$, since $E$ is connected.
 Since $\sin^2$ is a non-constant holomorphic function from ${\bf C}^d$ into itself then it is 
 an open map and $\sin_\T^2: \: \T_{\bf C}^d={\bf C}^d/(2\pi {\bf Z})^d\mapsto {\bf C}^d$ is also open.
 Let $V$ an open set of $\tilde E$ so $V=\tilde E\cap V'$ where $V'$ is an open set of $\T_{\bf C}^d$.
 Then $\phi_E(V)=\phi_E((\sin_\T^2)^{-1}(E)\cap V')=E\cap \sin_\T^2(V')$.
 Since $\sin_\T^2$ is an open map then $\sin_\T^2(V')$ is open.
 Thus $\phi_E(V)$ is an open set of $E$. Hence $\phi_E$ is an open map.
 Let us prove that $\phi_E$ is a closed map. 
 We observe that $|\sin^2(\Re x+i\Im x)|=\sinh^2(\Im x)+\sin^2(\Re x)$ so a set of the form $\sin_\T^2(A)$,
 $A\subset\T_{\bf C}^d$, is unbounded iff $A$ is unbounded. (In fact $\sin_\T^2$ is proper.)
 Let $F$ a closed subset of $\tilde E$. Let $z_n$ a sequence of values in $\phi_E(F)$ which tends to
 some $z\in E$. Then $z_n=\phi_E(x_n)$ tends to $z$, $x_n\in F$, so, by the above observation,
 the sequence $(x_n)\subset F^\N$ is bounded.
 Let $x'$ a subsequential limit of $x_n$. Then $z=\phi_E(x')$ so $z\in \phi_E(F)$.
 Hence $\phi_E(F)$ is closed.\\
 \indent The conclusion then follows. In fact, since $\phi_E$ is an open and closed map
 and $C_x$ is open and closed then $\phi_E(C_x)$ is both open and closed in the connected
 space $E$, so it coincides with $E$.
\end{proof}
\begin{lemma}
\label{l.54}
 Assume $\bfbeta=0$. Set $Q(x)=\Psi_0(z)-\lambda^2$ with $z=\sin^2 x$.
 Then $M^{\bf C}_{sgn}(Q)$ is discrete.
 In addition, assume $|\lambda|\in (0,\lambda_+)$.
 Then, each connected component of $M^{\bf C}_{reg}(Q)$ intersects $\T^3$ and
 the intersection is a two-dimensional real manifold.
\end{lemma}
\begin{proof}
 Since $\partial_{z_j}\Psi_0(z)\neq 0$ for all $z\in {\bf C}^3$, and since $\partial_{x_j}z=\sin(2x_j)e_j$
 vanishes iff $z_j\in \{0,1\}$, then $\nabla Q(x)=0$ iff $z\in \{0,1\}^3$.
 Hence $M^{\bf C}_{sgn}(Q)=  X_{0,1}$ is discrete.
 Let $C_x$ be a connected component of $M^{\bf C}_{reg}(Q)$. Set $P=\Psi_0-\lambda^2$.
 We have $M^{\bf C}_{reg}(Q)=(\sin_\T^2)^{-1}(E)$ with $E:= M^{\bf C}_{reg}(P)\sauf \{0,1\}^3$.
 Since $M^{\bf C}_{reg}(P)$ is a connected two-dimensional complex manifold, then so is $E$.
 Since $E$ is a two-dimensional complex manifold, so is $C_x$.
 Thanks to Lemma \ref{l.topo1}, we obtain $\sin^2(C_x)=E$. 
 Since $E\subset M^{\bf C}_{reg}(P)$ and $M^{\bf C}_{reg}(P)$ intersects $(0,1)^3$,
  $\sin^2(C_x)$ intersects $(0,1)^3$. Hence $C_x$ intersects $\T^3$.
 Finally, since the jacobian of $\sin^2|_{\T^3}$ does not vanish on $(\sin_\T^2)^{-1}((0,1)^3)$
 ($\subset \T^3$), then $C_x\cap \T^3$ is a two-dimensional real manifold.
\end{proof}
\begin{lemma}
\label{l.55}
 Assume $\bfbeta\neq 0$ and $\Pi_{j=1}^3\beta_j=0$.
 Set $Q(x)=q(x;\lambda)$.\\
 (1) The analytic variety $M^{\bf C}_{sgn}(Q)$ has Hausdorff ($5$)-measure zero and $M^{\bf C}_{sgn}(Q)\cap \T^3$ is finite.\\
 (2) We have $M^{\bf C}_{reg}(Q)=M_{\bf C}^+ \cup M_{\bf C}^-$ where each $M_{\bf C}^\pm$, defined
 by
\[ M_{\bf C}^\pm :=    \{x\in \T^3_{\bf C} \; \mid \; z\not\in \{0,1\}^3, \; \tau^\pm(z)=\lambda^2 \neq \tau^\mp(z)\}\]
 is a two-dimensional submanifold of $\T_{\bf C}^3$.\\
 (3) If $|\lambda|\in (0,\lambda_-)$, each connected component of $M_{\bf C}^\pm$
 intersects $\T^3$ and the intersection is a two-dimensional real manifold.\\
 (4) If $|\lambda|\in [\lambda_-,\lambda_+)$, each connected component of $M_{\bf C}^+$
 intersects $\T^3$  and the intersection is a two-dimensional real manifold.
 However, the set $M_{\bf C}^- \cap \T^3$ is finite.
\end{lemma}
\begin{proof}
 We use the notations in Lemma \ref{l.45}, so $Q(x)=P(z)$, and we set 
 $Q^\pm(x)=\tau^\pm(z) - \lambda^2=P^\pm(z)$.\\
(1) We have
\begin{eqnarray}
\nonumber
 M^{\bf C}_{sgn}(Q) &=& (M^{\bf C}(Q)\cap X_{0,1}) \cup (M^{\bf C}(Q^+)\cap M^{\bf C}(Q^-)) \\
\label{val3.MCsgn}  &=& (M^{\bf C}(Q)\cap X_{0,1}) \cup (\sin_\T^2)^{-1}(M^{\bf C}_{sgn}(P)).
\end{eqnarray}
 The set $M^{\bf C}(Q)\cap X_{0,1}$ is finite since $X_{0,1}$ is finite.
 In addition, the set $(\sin_\T^2)^{-1}(M^{\bf C}_{sgn}(P))$ has Hausdorff $k$-measure zero for all $k\ge 3$,
 since $M^{\bf C}_{sgn}(P)$ is a complex straight line (see Lemmas \ref{l.3}, \ref{l.45}) and $\sin_\T^2$
 is a local smooth diffeomorphism except on a finite set of ${\bf C}^3$.
 Hence $M^{\bf C}_{sgn}(Q)$ has Hausdorff $(2d-1=5)$-measure zero.
 Since $M^{\bf C}_{sgn}(P)\cap [0,1]^3$ has at most one point (see Lemma \ref{l.45}),
  \eqref{val3.MCsgn} shows that $M^{\bf C}_{sgn}(Q)\cap \T^3$ is finite. This proves (1).\\
(2) The relation $M^{\bf C}_{reg}(Q)=M_{\bf C}^+ \cup M_{\bf C}^-$ is then obvious.
 In addition, we have $M_{\bf C}^\pm=(\sin_\T^2)^{-1}(M^{\bf C}(P^\pm)) \sauf M^{\bf C}_{sgn}(Q)$.
 Since $M^{\bf C}(P^\pm)$ is a two-dimensional submanifold of ${\bf C}^3$ and $M^{\bf C}_{sgn}(Q)$
 has dimension one,  $M_{\bf C}^\pm$ is a two-dimensional submanifold of $\T_{\bf C}^3$.\\
 (3) and (4) for $M_{\bf C}^+$.  It is similar to the corresponding assertion in  Lemma \ref{l.54}.\\
 (4) for $M_{\bf C}^-$. This is a direct consequence of (3-3), Lemma \ref{l.45}.
\end{proof}
\begin{lemma}
\label{l.56}
 Assume $\Pi_{j=1}^3\beta_j\neq0$.
 Let $|\lambda| \in (0,\lambda_+)$ and set $Q(x)=\lambda^{-2}p(z;\lambda)$.\\
(1) The set $M^{\bf C}_{sgn}(Q)$ is finite. \\
 (2) Assume $|\lambda| \in (0,\lambda_+)$.
 Then, each connected component of $M^{\bf C}_{reg}(Q)$ intersects $\T^3$,
 and the intersection is a two-dimensional real manifold.
\end{lemma}
\begin{proof}
 Put $P(z):=Q(x)$ so $\partial_{x_j}Q(x)=0$ iff $\partial_j P(z)=0$ or $z_j\in \{0,1\}$.
 Hence $M^{\bf C}_{sgn}(Q) = (\sin_\T^2)^{-1}(F)$ with
\begin{eqnarray*}
 F &:=& M^{\bf C}_{sgn}(P)  \cup_{j=1}^3 F_j \cup_{j\neq k} F_{j,k} \cup (\{0,1\}^3\cap M^{\bf C}(P)),\\
 F_j &:=& \{z\in M^{\bf C}(P)\; \mid \; z_j\in\{0,1\}^3, \; \partial_{z_k}P(z)=0 \quad k\neq j\},\\
 F_{j,k} &:=& \{z\in M^{\bf C}(P)\; \mid \; z_j,z_k\in\{0,1\}^3, \; \partial_{z_l}P(z)=0 \quad l\neq j,k\}.
\end{eqnarray*}
(1) Thanks to Lemma \ref{l.46} we have $M^{\bf C}_{sgn}(P)=\emptyset$.
 We denote by $A_j$ the $j^{\mathrm{th}}$ column of the matrix $A$ defined 
 by  \eqref{def.A}, by $E_j$ the column of the coefficients of $e_j$ in the canonical basis $(e_1,e_2,e_3)$,
 by $\tilde A_j$ the $3\times 3$ matrix obtained from $A$ by replacing the column $A_j$ by $E_j$
 and by $\tilde A_{j,k}$ (with $j\neq k$) the $3\times 3$ matrix obtained from $A$ by replacing the columns
 $A_j$ and $A_k$ respectively by the column $E_j$ and $E_k$.
 Since $P(z) = Az\cdot z + \bfb \cdot z + c$, then $\partial_j P(z)= 2A_j\cdot z + \bfb_j$.\\
 Let us consider $F_1$. 
 We have $F_1=F_1(0)\cup F_1(1)$, where
 $F_1(\xi)$ is the intersection of the three hyperplanes $z_1=\xi$, $2A_k\cdot z + \bfb_k=0$, $k=2,3$.
 The matrix of the above system is $B_1:=(E_1|A_2|A_3)$ where $E_1:=(1\: 0\: 0)^T$.
 Its determinant is
\[
 \det(B_1)=\gamma_2\gamma_3-(\epsl_1\mu_1\alpha_1)^2 =
  (\gamma_1-\alpha_1^2)(\epsl_1\mu_1)^2
 =  -\frac14 \beta_1^2  (\epsl_1\mu_1)^2 <0.
\]
 Hence $F_1(\xi)$ is reduced to one point, and then $F_1$ is a couple of points.
 Similarly, each $F_j$ is reduced to two points. 
 Let us consider $F_{1,2}$. We have $F_{1,2}=\cup_{\xi,\xi'\in \{0,1\}} F_{1,2}(\xi,\xi')$
 where $F_{1,2}(\xi,\xi')$ is the intersection of the three hyperplanes $z_1=\xi$, $z_2=\xi'$, $2A_3\cdot z + \bfb_3=0$.
 The matrix of the above system is $B_{1,2}:=(E_1|E_2|A_3)$ where $E_2:=(0\: 1\: 0)^T$.
 Its determinant is $\det(B_{1,2})=\gamma_3\neq 0$.
 Hence $F_{1,2}$ is finite.
 Then, $F$ is finite. Now we observe that, for any finite set $E\subset{\bf C}^3$, $(\sin_\T)^{-1}(E)$
 is a finite subset of $\T_{\bf C}^3$. Consequently, $M^{\bf C}_{sgn}(Q)$ is finite.\\
(2) We have $M^{\bf C}_{reg}(Q) = (\sin_\T^2)^{-1}(E)$ with  $E := M^{\bf C}_{reg}(P) \sauf F$.
 Thanks to Lemma \ref{l.46}, $M^{\bf C}_{reg}(P)$ is a  connected two-dimensional submanifold of
 ${\bf C}^3$.
 So, since $\dim F \le 1$, then $E$ is also a connected two-dimensional submanifold of ${\bf C}^3$.
 Thus $M^{\bf C}_{reg}(Q)$ is a two-dimensional submanifold of $\T_{\bf C}^3$.
 Let $C_x$ be a connected component of $M^{\bf C}_{reg}(Q)$. Lemma \ref{l.topo1} says that
 $\sin_\T^2(C_x)=E$.
 We have $E\cap (0,1)^3= M^{\bf C}_{reg}(P)\cap(0,1)^3$, since $F \cap(0,1)^3$ is empty.
 Thanks to Lemma \ref{l.46}, $M^{\bf C}_{reg}(P)\cap(0,1)^3$ has dimension two.
 Thus $E$ intersects $(0,1)^3$. 
 The end of the proof is similar to the end of the proof of Lemma \ref{l.54}.
\end{proof}

\section{Proof of the main Theorems}
\label{sec.Proof-main}
\subsection{Proof of Theorem \ref{th.2}}
 Theorem \ref{th.2} is a straight consequence of Theorem \ref{th.31}, Lemmas \ref{l.54}, \ref{l.55}, \ref{l.56}.
 
\subsection{Proof of Theorem \ref{th.3}}
 We need
\begin{lemma}
\label{S5Lemma5.1}
 Assume $\beta \neq 0$ and $\beta_1\beta_2\beta_3 =0$. Let $\lambda \neq 0$. \\
\noindent
 (1) The set $M^{\bf C}_{reg}(q(\cdot;\lambda))$ is a disjoint union of the following
 two complex manifolds of dimension 2:
\[
 M_{\bf C}^{\pm} := \{x \in {\mathbb T}^3_{\bf C} \; \mid \; z \not\in \{0,1\}^3, 
 \tau^{\pm}(z) = \lambda^2 \neq \tau^{\mp}(z)\}.
\]
 (2) Assume $|\lambda| > \lambda_-$. Then, $M_{\bf C}^-\cap \mathbb T^3$ is finite,
\end{lemma}

 The assertions (1) and (2) of Lemma \ref{S5Lemma5.1} are a straight consequence of Lemma \ref{l.55}.\\
 
 Set
\[ B(x)=- \bfmu \tilde M(y) \bfepsl \tilde M(y), \]
 and $u=(u_E,u_H)$ with
\[
 u_E(x) := \frac1{\lambda} \bfepsl \tilde M(y) u_H(x)  , \quad
 u_H(x) := (\tau^-(z)-\lambda^2)^{-1} \Com(B(x)-\lambda^2) v_H, 
\]
 where $v_H$ is a non-null constant column-vector of length 3.
 Since $u_H\in \calC^\infty(\T^3;{\bf C}^3)$ so does $u_E$. Thus $\hat u\in L^2({\bf Z}^3)\subset \calB_0^*({\bf Z}^3)$.
 Then, $(B(x)-\lambda^2)u_H(x)= \lambda^2 (\tau^+(z)-\lambda^2)v_H$ is a trigonometric polynomial,
 and so is $(H^D-\lambda)u=(0,\lambda^{-1}(B-\lambda^2)u_H)$.\\
 Let us prove that we can choose $v_H$ such that $u$ is not a trigonometric polynomial.
 We remember that the family of eigenvalues of the symmetric matrix $-\bfepsl \tilde M(y) \bfmu \tilde M(y)$
 is $(0,\tau^+(z),\tau^-(z))$ and we denote by $\Pi^0(y)$, $ \Pi^+(y)$, $\Pi^-(y)$, respectively,
 the associated spectral eigenprojectors.
 Thus, we have
\begin{eqnarray*}
 \Com(B(x)-t) &=& -t(\tau^+(z)-t) \Pi^-(y) -t(\tau^-(z)-t) \Pi^+(y)\\
 && +(\tau^+(z)-t)(\tau^-(z)-t)\Pi^0(y),
\end{eqnarray*}
 therefore, 
\begin{equation}
\label{rel.comB}
 \Com(B(x)-\tau^-(z)) = -2\tau^-(z) \sqrt{K_0(z)} \Pi^-(y).
\end{equation}
 Hence, $\Com(B(x)-\tau^-(z))$ has rank one at any $z\in {\bf C}^3$ such that $K_0(z)\neq 0$.
 If each coefficient $c_{j,k}(y)$ of $\Com(B(x)-\lambda^2)$ were reducible by $(\tau^-(z)-\lambda^2)$,
 i.e, if  $(\tau^-(z)-\lambda^2)^{-1}c_{j,k}(y)$ were a trigonometrical polynomial, then
 $\Com(B(x)-\lambda^2)$ would vanish at any $z\in {\bf C}^3$ such that $\tau^-(z)-\lambda^2=0$.
 This is in contradiction with \eqref{rel.comB}.
 Hence there exists $v_H$ such that $u$ is not a trigonometrical polynomial.
 This ends the proof of Theorem \ref{th.3}.

 We complete this section by proving the assertion of Remark \ref{rem1.th3}.
 Set $Q^\pm(x)=(\tau^\pm(z)-\lambda^2)v(x)$, $v:=Q^-u$. Since $\tau^-$ is linear, $Q^-$ is then smooth
 in the $x$ variable, so $v\in \calB_0^*(\T^3)$.
 In addition we have $Q^+ = \lambda^{-2} q(\cdot;\lambda)u$ which is a trigonometric polynomial.
 Since $Q^+$ satisfies Conditions (A-1)-(A-2) of Theorem \ref{th.31}, the result follows.

\subsection{Proof of Theorem \ref{th.1.3}}
 I. We put
$$
 \calP_{disc}^-(\xi,l):=  \{n\in{\bf Z}^3\; \mid \; \xi \cdot n \le l\} ,
 \quad (\xi,l)\in {\bf R}^3\times \R,
$$ 
 and
\[
 G^* :=\{\xi \in \Sm^2\; \mid \; |\xi_i|> |\xi_j|\ge |\xi_l| \mbox{ for some permutation $(i,j,l)$
 of $(1,2,3)$}\},
\]
 where $\Sm^2$ denotes the euclidian unit sphere of ${\bf R}^3$.
 By definition we have $\Kint=\Om\cap{\bf Z}^3$ where $\Om$ is a bounded convex set of ${\bf R}^3$.
\begin{lemma}
\label{l.charact-Kint}
 There exists a finite family $F\subset G^* \times \R$ such that
\begin{equation}
\label{val.Kint}
 \Kint= \cap_{(\xi,l)\in F} \calP_{disc}^-(\xi,l) .
\end{equation}
\end{lemma}
 Proof. Notice that the characterization \eqref{val.Kint} of $\Kint$ is well-known
 with $G^*$ replaced by $\Sm^2$.
 The bounded convex set $\Om\subset {\bf R}^3$ can be assumed closed and written 
 $\Om = \cap_{(\xi,r)\in F_0} \calP^-(\xi,r)$, where $\calP^-(\xi,r):=\{\nu\in{\bf R}^3$ $\mid$ $\nu \cdot \xi \le r\}$
 and $F_0\subset \Sm^2 \times \R$.
 Since $\Kint$ is bounded then $\Kint\subset K_1$, $K_1:= [-l_0,l_0]^3\cap {\bf Z}^3$ for some $l_0>0$.
 Since $\Kint$ is finite we then have
\[ \Kint= \cap_{(\xi,r)\in F_1} \calP^-(\xi,r)  \cap K_1,\]
 where $F_1\subset F_0$ is a finite family.
 Obviously, we can assume that $F_1$ is non void and that $K_1 \not\subset \calP^-(\nu,r)$, $(\nu,r)\in F_1$.
 (If $K_1\subset \calP^-(\nu,r)$ we suppress $(\nu,r)$ in $F_1$, that is, we replace $F_1$ by $F_1\sauf\{(\nu,r)\}$.)
 Fix $(\nu,r)\in F_1$. 
 Since $K_1\sauf  \calP^-(\nu,r)$ is finite and non void, we put
 $\delta:= \sup\{ \nu \cdot n -r$ $\mid$ $n\in K_1\sauf  \calP^-(\nu,r)\}>0$.
 We then have
\[
  \calP^-(\nu,r)  \cap K_1 =\calP^-(\nu,r+\delta/2)  \cap K_1 , 
\]
\[ 
 K_1\sauf  \calP^-(\nu,r) = K_1\sauf  \calP^-(\nu,r-\delta/2).
\]
 Since $G^*$ is dense in $\Sm^2$ and since $K_1$ is finite, there then exists $\xi\in G^*$
 sufficiently closed to $\nu$ such that
\[
\begin{aligned}
\calP^-(\xi,r)  \cap K_1 \subset \calP^-(\nu,r+\delta/2)  \cap K_1,\\
  K_1\sauf  \calP^-(\xi,r) \subset K_1\sauf  \calP^-(\nu,r-\delta/2).
\end{aligned}
\]
 Then we obtain
\[ \calP^-(\nu,r) \cap K_1 = \calP^-(\xi,r)\cap K_1 =  \calP_{disc}^-(\xi,r)\cap K_1.\]
 Hence, \eqref{val.Kint} holds with $F=F_1 \cup_{j=1}^3 (e_j,l_0) \cup_{j=1}^3 (-e_j,l_0)$.
 \qed

\medskip
 II. Letting $\hat v:\:{\bf Z}^3 \ni n \mapsto \hat v(n) \in {\bf C}^d$ be a sequence with compact support,
 we set
\[ \hat v_S(n)=\ove{\hat v(n)},\]
 so $v_S(x) = \ove{v(-x)}$. In addition, letting $\xi\in{\bf R}^3$, $\xi\neq 0$, we define
\[ N_\xi^{max}(v) = N_\xi^{max}(\hat v) := \max\{n\cdot\xi \; \mid \; \hat v(n)\neq 0\} ,\]
 with the convention $\max(\emptyset)=-\infty$.
 
 We put $|\xi|_\infty:=\max(|\xi_j|;\; j\in\{1,2,3\})$, where the $\xi_l$'s are the usual coordinates
 of $\xi$.  
\begin{lemma}
\label{lem.4}
 Let $\lambda\neq 0$ and $\xi\in G^*$. Let $v$ and $w$ are two trigonometric polynomials.

 1) We have
$$ N_\xi^{max}(v+w) = N_\xi^{max}(v) \quad {\rm if} \quad N_\xi^{max}(v)> N_\xi^{max}(w),$$
$$ N_\xi^{max}(vw) \le N_\xi^{max}(v) + N_\xi^{max}(w) $$
 holds true.

2) We have
\[
 \begin{aligned}
  N_\xi^{max}(e^{ix_j})&=N_\xi^{max}(\delta_{e_j}) = |\xi_j|, \\
  N_\xi^{max}(z_j) &= 2|\xi_j| ,\\
  N_\xi^{max}(z_j^2) &= 4|\xi_j|.
\end{aligned}
\]
3) For $t\neq 0$ and all sequence $\hat v$ we have
\begin{equation}
\label{val.Nqv}
 N_\xi^{max}(q_t v) = 4|\xi|_\infty + N_\xi^{max}(v).
\end{equation}
4) We have
\begin{eqnarray*}
 N_\xi^{max}(v_S) &=& N_\xi^{max}(v),\\
 N_\xi^{max}(<v,\ove{v_S(x)}>_{{\bf C}^m}) &=& 2N_\xi^{max}(v) .
\end{eqnarray*}
\end{lemma} 
 Proof.
1) Obvious.\\
2) Since $z_j= (-4)^{-1}(e^{2ix_j}+e^{-2ix_j}-2)$, the computation is straightforward.\\
3) Thanks to \eqref{form.pz} and \eqref{def.A} we have 
\[ t^{-2} q_t(x) = \sum_{j=1}^3 \gamma_j z_j^2 + \sum_{j\neq l} A_{j,l}z_j z_l + \bfb\cdot z+ t^2 ,\]
 with $t\gamma_j\neq 0$.
 Since $|\xi|_\infty = \xi_{i^*}>\max(|\xi_j|,|\xi_k|)$ for some $i^*$ where $\{i^*,j,k\}=\{1,2,3\}$,
 then,
\[ N_\xi^{max}(q_t) = N_\xi^{max}(z_i^2)= 4\xi_{i^*} = 4|\xi|_\infty.\]
 In addition, we then have
$$ N_\xi^{max}(q_t v) \le  N_\xi^{max}(q_t) + N_\xi^{max}(v) = 4|\xi|_\infty + N_\xi^{max}(v). $$
%
%
 We have
\begin{eqnarray*}
 N_\xi^{max}(e^{\pm  2ix_j} v) &=& \max\{n\cdot\xi \; \mid \; U(e^{\pm 2ix_j} v)(n)\neq 0\}
  =  \max\{n\cdot\xi \; \mid \; v(n\pm 2e_j)\neq 0\} \\
 &=&  \max\{(n\mp 2e_j)\cdot\xi \; \mid \; v(n)\neq 0\} = N_\xi^{max}(v) \mp 2 \xi_j.
\end{eqnarray*}
 Hence, $N_\xi^{max}(z_jv) =  N_\xi^{max}(v) + 2|\xi_j|$,
 $N_\xi^{max}(z_jz_k v) =  N_\xi^{max}(v) + 2|\xi_j| + 2|\xi_k|$.
 Thus, since $\xi\in G^*$,
$$ N_\xi^{max}(z_jz_k v) < N_\xi^{max}(v) + 4|\xi| \quad {\rm if}\quad (j,k)\neq (i^*,i^*).$$
 Hence,
$$ N_\xi^{max}(q_tv) =  N_\xi^{max}(z_{i^*}^2 v) = N_\xi^{max}(v) + 4|\xi| .$$
4) The function $<v(x),\ove{v_S(x)}>_{{\bf C}^d}=(2\pi)^{-3}\sum_{n,m}e^{-i(n+m)x}\hat v(n)\ove{\hat v(m)}$
 is a trigonometric polynomial, and, by a simple computation, the conclusion follows.
\qed 

 
 \medskip
III. We put $L_t(x)=L(x;t):=\Com(H^D(x)-t)$, and $q_t(x):=\det(H^D(x)-t)$.
 The symmetric real matrix $L(x;t)$ has size $5\times 5$ and $L$ is a homogeneous polynomial
 of degree 5 in the variables $t,y_1,y_2,y_3$.
 Since $H^D(x)$ has rank $\le 4$ then, $L_0=0$ and, consequently,
 $L_t$ is a polynomial of the variables $y_1,y_2,y_3$ with degree $\le 4$.
 Hence $L_t$ has the form
\begin{equation}
\label{rel.B1vS}
 L_t(x) = tB_1(x) + t^2 B_2(x;t),
\end{equation}
 where $B_1(x)=-\Phi_0(z)\Pi_0(y)$, $\Pi_0(y)$ is the orthogonal projection on $\ker H^D(x)$,
 $\Phi_0:=\Psi_0^2-K_0$ is a homogeneous polynomial of degree 2 of the variables $z_1,z_2,z_3$,
 and $B_2(\cdot;t)$ is a trigonometric polynomial of degree 3.
 Thus, if $t\neq0$, then $L_t$ is a trigonometric polynomial of degree $4$.
 Since $B_1(x)=\lim_{t\to 0} t^{-1}L_t(x)$ we then see that $B_1$ is also a trigonometric polynomial
 of degree 4.\\
 We have
\begin{lemma}
\label{l.Nmax(q)}
 Let $t\neq 0$ and $\xi\in G^*$. We then have
\begin{equation}
\label{est.NLt}
 N_\xi^{max}(L_t) = N_\xi^{max}(B_1) = N_\xi^{max}(q_t) = 4|\xi|_\infty >  3|\xi|_\infty
  \ge  N_\xi^{max}(B_2(t)).
\end{equation}
\end{lemma}
\begin{proof}
 Without loss of generality we assume $\xi_1>|\xi_j|$, $j\in\{2,3\}$.
 Since $q_t I = (H^D-t)L_t$, we then have, thanks to Lemma \ref{lem.4},
\begin{equation}
\label{rel.qtHDLt}
 4|\xi|_\infty = 4\xi_1 = N_\xi^{max}(q_t) \le N_\xi^{max}(H^D) +N_\xi^{max}(L_t).
\end{equation}
 Since $H^D-t$ (respect., $L_t$) is a polynomial of degree $1$ (respect;, of degree 4) of
 the variables $y_j=(2i)^{-1}(e^{ix_j}-e^{-ix_j})$, then $N_\xi^{max}(H^D-t)\le \xi_1$ and
 $N_\xi^{max}(L_t)\le 4\xi_1$. Hence the inequalities in \eqref{rel.qtHDLt} are equalities.
 Finally, since $B_2(t)$ is a polynomial of degree $\le 3$ of the variables $y_j=(2i)^{-1}(e^{ix_j}-e^{-ix_j})$,
 then $N_\xi^{max}(B_2(t))\le 3\xi_1$.
\end{proof}

\noindent IV. We have the following
\begin{lemma}
\label{lem.5}
 Let $\lambda\neq 0$ and $\xi\in G^*$. Let $l\in\R$ and $\hat u$ be compactly supported in ${\bf Z}^3$
 such that $(\hat H^{D_p}-\lambda)\hat u(n)=0$ in $\{n\in{\bf Z}^3$ $\mid$ $\xi\cdot n>l\}$.
 Then, $\hat u(n)=0$ in $\{n\in{\bf Z}^3$ $\mid$ $\xi\cdot n>l\}$.
\end{lemma} 
\begin{proof}
 It is sufficient to consider the case $l=0$.
 The sequence $\hat u$ satisfies
\begin{equation}
\label{eq.HDpvu=f}
 (H^{D_p}-\lambda) u= f,
\end{equation} 
 where $\hat f$ is compactly supported in $\{n\in{\bf Z}^3$ $\mid$ $\xi\cdot n\le 0\}$.
 Thus, $N_\xi^{max}(f)\le 0$. 
 The operator $\hat D(\hat D_p)^{-1}$ is bounded and invertible in $l^2({\bf Z}^3)$ and
 is an operator of multiplication by a diagonal matrix of the form $I+\hat K$, so
 $\hat K=\hat D(\hat D_p)^{-1} - I$ is an operator of multiplication by a diagonal matrix
 and has compact support.
 We then have
\begin{equation}
\label{eq.HDu}
 (H^D -\lambda) u = \lambda  K u + g,
\end{equation}
 where $\hat g:=\hat D(\hat D_p)^{-1}\hat f$ vanishes in $\{n\in{\bf Z}^3$ $\mid$ $n\cdot\xi>0\}$,
 so $N_\xi^{max}(g)\le 0$.\\
 Since $\hat D_p$ is real-valued we have
\[ D_p v_S(x) =  (2\pi)^{-3/2}\sum_{n\in{\bf Z}^3} \hat D_p(n)\ove{\hat v(n)}e^{-inx} = \ove{D_p v(-x)} .\]
 In addition, since $\ove{H^D(-x)} = H^D(-x) =-H^D(x)$, then \eqref{eq.HDpvu=f} implies
\[
 -D_p H_0 u_S = \lambda u_S + f_S.
\]
 Since $B_1(y)=-\Phi_0(z)\Pi_0(y)$ we have $B_1 H_0=0$. 
 Then, multiplying \eqref{rel.B1vS} from the left by $B_1 D_p^{-1}$ and observing that
 $\Pi_0 H_0=0$, we get
\begin{equation}
\label{rel2.B1v}
  B_1D_p^{-1} u_S =  -\lambda^{-1} B_1 D_p^{-1} f_S .
\end{equation}
 Multiplying \eqref{eq.HDu} from the left by $L_\lambda$ and using \eqref{rel.B1vS}, we get
\[
  q_\lambda u = \lambda L_\lambda K u + L_\lambda g = \lambda^2 B_1K u + R_0,
\]
 where
\begin{equation}
\label{def.R0}
  R_0:=\lambda^3 B_2(\cdot;\lambda)K u + L_\lambda g.
\end{equation}
 Put $T(x):=\sum_{n\in{\bf Z}^3} \hat D_p^{-1} (n) |\hat u(n)|^2 e^{2inx}$.
 On one hand, since $\hat D_p^{-1} (n) |\hat u(n)|^2\neq 0$ if $\hat u(n)\neq 0$, then
\begin{equation}
\label{val.NT}
  N_\xi^{max}(T) = 2N_\xi^{max}(u).
\end{equation}
 On the other hand, we have
\[
  T(x) =  < D_p^{-1}u_S(x),u(x)>_{{\bf R}^6}.
\]
 Thus,
\begin{eqnarray*}
  q_\lambda(x) T(x)  &=&  < D_p^{-1}u_S(x),q_\lambda u(x)>_{{\bf R}^6} \\
                     &=& \lambda^2 < D_p^{-1}u_S(x),B_1(x) K u(x)>_{{\bf R}^6} + R(x),
\end{eqnarray*}
 where $R(x):= < D_p^{-1}u_S(x),R_0(x)>_{{\bf R}^6}$.
 Then, since $B_1(x)$ and $D$ are real-symmetric matrices, by virtue of \eqref{rel2.B1v},
\begin{eqnarray}
\nonumber
  q_\lambda(x) T(x) &=& \lambda^2 < B_1D_p^{-1}u_S(x),K u(x)>_{{\bf R}^6} + R(x)\\
\label{val.qT}
  &=& -\lambda < B_1 D_p^{-1} f_S(x),K u(x)>_{{\bf R}^6} +R(x).
\end{eqnarray}
 We have, in view of \eqref{def.R0}, \eqref{est.NLt},
\begin{eqnarray}
\nonumber
  N_\xi^{max}(R) \le N_\xi^{max}(D_p^{-1}u_S) +
  \max(N_\xi^{max}(B_2(\cdot;\lambda)K u),N_\xi^{max}(L_\lambda g)) \\
\label{maj.NR}
 \le N_\xi^{max}(u)+ \max(3|\xi|_\infty+N_\xi^{max}(u),4|\xi|_\infty),
\end{eqnarray}
 and, in addition,
\begin{eqnarray}
\label{maj.NB1DpKu}
  N_\xi^{max}(< B_1 D_p^{-1} f_S,K u>_{{\bf R}^6} ) \le 4|\xi|_\infty + N_\xi^{max}(u).
\end{eqnarray}
 Thus, in view of \eqref{val.qT}, \eqref{maj.NR}, \eqref{maj.NB1DpKu}, and \eqref{est.NLt},
\[
 N_\xi^{max}(q_\lambda T)  \le  \max(4|\xi|_\infty + N_\xi^{max}(u), 3|\xi|_\infty+2N_\xi^{max}(u)).
\]
 Using \eqref{val.Nqv}, \eqref{val.NT} and \eqref{est.NLt}, we obtain
\begin{eqnarray}
\nonumber
 N_\xi^{max}(T) = N_\xi^{max}(q_\lambda T) - N_\xi^{max}(q_\lambda) \\
\label{maj.T}
 \le  \max(N_\xi^{max}(u), -|\xi|_\infty+2N_\xi^{max}(u)).
\end{eqnarray}
 Thus, \eqref{maj.T} and \eqref{val.NT} show that
\[
   N_\xi^{max}(u) = N_\xi^{max}(T)-N_\xi^{max}(u) \le 0.
\]
\end{proof}

\noindent V. (Last step.) Theorem \ref{th.1.3} is a direct consequence of Lemma \ref{lem.5}.
 In fact, putting $H(\xi,l):=\{n\in{\bf Z}^3$ $\mid$ $\xi\cdot n> l\}$, $l\in\R$, we have
 $H(\xi,l)= {\bf Z}^3\sauf \calP_{disc}^-(\xi,l)$.
 Hence, if $u$ satisfies the conditions of Theorem \ref{th.1.3}, then $\hat u(n)=0$ in 
 $\cup_{(\xi,l)\in F} H(\xi,l)= {\bf Z}^3\sauf \Kint=\Kext$.
\qed

\subsection{Proof of Proposition \ref{prop.1}}
 As above we put $\hat K=\hat D(\hat D_p)^{-1} - I$, which has compact support.
 The distribution $u$ satisfies the above relation \eqref{eq.HDu} with $g=0$.
 Hence, the distribution $v\in \calB_0^*(\T^3)$ defined by $v(x)=(\tau^-(z)-\lambda^2)u(x)$
 is a trigonometric polynomial. By observing that $\tau^-(z)-\lambda^2\le \lambda_-^2-\lambda^2<0$,
 we thus obtain $u(x)= (\tau^-(z)-\lambda^2)^{-1}v(x)$, so $u\in L^2(\T^3,{\bf C}^6)$.

\subsection{Proof of Corollary \ref{coro.1.4}}
 Put $f:=(H^{D_p}-\lambda)u$.
 Assume that a finite convex set $\Kint$ contains $\supp(\hat D^p-\hat D)$.
 Then, $\hat f$ has support in $\Kint$.
 In addition, we have $(\hat H^D-\lambda)\hat u=\lambda \hat K \hat u + \hat D^{-1} \hat D_p \hat f$,
 where $\hat K:= \hat D (\hat D_p)^{-1}-I$ has compact support.
 Since (RT) for $\hat H^D-\lambda$ holds, then $\hat u$ has compact support. 
 Consequently, thanks to Theorem \ref{th.1.3}, $\hat u$ vanishes in $\Kext$.
\qed

\end{document}